\documentclass[11pt,a4paper,reqno]{amsart}

\usepackage[T1]{fontenc}
\usepackage{lmodern}
\usepackage{amssymb}
\usepackage{enumerate,microtype}
\usepackage[top=2.7cm,bottom=2.7cm,left=2cm,right=2cm]{geometry}
\usepackage[hidelinks]{hyperref}

\newtheorem{Theorem}{Theorem}[section]
\newtheorem{lemma}[Theorem]{Lemma}
\newtheorem{Prop}[Theorem]{Proposition}
\newtheorem{Cor}[Theorem]{Corollary}
\newtheorem{mainthm}{Theorem}

\newtheorem{maincor}[mainthm]{Corollary}
\theoremstyle{definition}
\newtheorem{remark}[Theorem]{Remark}
\newtheorem{question}[Theorem]{Question}

\newtheoremstyle{citedremarkstyle}
{6pt plus 2pt minus 2pt}{6pt plus 2pt minus 2pt}
{\normalfont}{}{\bfseries}{.}{.5em}
{\thmname{#1}\thmnumber{ #2}\thmnote{ \normalfont#3}}
\theoremstyle{citedremarkstyle}
\newtheorem{citedremark}[Theorem]{Remark}
\newtheoremstyle{citedtheoremstyle}
{}{}
{\itshape}{}{\bfseries}{.}{.5em}
{\thmname{#1}\thmnumber{ #2}\thmnote{ \normalfont#3}}
\theoremstyle{citedtheoremstyle}
\newtheorem{citedtheorem}[Theorem]{Theorem}
\newtheorem{citedprop}[Theorem]{Proposition}
\newtheorem{citedlemma}[Theorem]{Lemma}
\theoremstyle{remark}
\newtheorem*{note}{Note}

\newcommand{\C}{\mathbb C}
\newcommand{\N}{\mathbb N}
\newcommand{\m}{\mathfrak m}
\DeclareMathOperator{\hgt}{ht}
\DeclareMathOperator{\depth}{depth}
\DeclareMathOperator{\Spec}{Spec}
\DeclareMathOperator{\rank}{rank}
\numberwithin{equation}{section}

\makeatletter
\renewenvironment{proof}[1][\proofname]{\par\pushQED{\qed}\normalfont\topsep6\p@\@plus6\p@\relax
  \trivlist\item[\hskip\labelsep\bfseries #1\@addpunct{.}]\ignorespaces}{\popQED\endtrivlist\@endpefalse}
\makeatother

\hypersetup{
	pdftitle={Prime ideals and regular sequences of power sum symmetric polynomials},
	pdfauthor={Neeraj Kumar}
}

\begin{document}

	\title[Prime ideals and regular sequences]
	{Prime ideals and regular sequences of power sum symmetric polynomials}
	\author{Neeraj Kumar}
	\address{Department of Mathematics, Indian Institute of Technology Hyderabad, Kandi, Sangareddy - 502284, INDIA}
	\email{neeraj@math.iith.ac.in}
	\date{November 26, 2026}

	\begin{abstract}
		\normalsize
		We study the power sum symmetric polynomials $p_j=x_1^j+\cdots+x_n^j$ in the polynomial ring
		$\C[x_1,\ldots,x_n]$; we call them power sums.
		We ask when a set of power sums forms a regular sequence, and when the ideal generated by these power sums
		is prime.
		For integers $d\geq1$ and $1\leq m\leq n-2$, we show that the ideal $(p_d,p_{2d},\ldots,p_{md})$ is prime,
		and we determine all degrees $c$ for which $p_d,p_{2d},\ldots,p_{md},p_c$ is a regular sequence.
		The answer is given by a numerical semigroup.
		In four variables, we show that the ideal $(p_d,p_{sd})$ is prime if and only if $s$ is even and, in addition,
		$s\ne4$ or $d$ is odd. For $s=3,4,5$ we find the minimal primes of this ideal.
		This proves a statement on the ideals $(p_3,p_b)$ that was conjectured by Conca, Singh and Soundararajan.
		We also determine all degrees $c$ for which $p_d,p_{kd},p_c$ is a regular sequence in four variables,
		for $k=2,3,4,5$.
		The proofs combine Newton's identities and the Jacobian criterion with the geometry of the plane curve
		$p_1=p_s=0$ and Kummer theory.
	\end{abstract}

	\maketitle

	\section*{Introduction}

	Let $S=\C[x_1,\ldots,x_n]$ be the polynomial ring in $n$ variables
	over the complex numbers, with its standard grading.
	For an integer $a\geq1$, let $p_a=x_1^a+\cdots+x_n^a$ be the power sum symmetric polynomial of degree $a$.
	For short, we call $p_a$ the power sum of degree $a$.
	Let $a_1,\ldots,a_r$ be distinct positive integers with $r\leq n$, and put $I=(p_{a_1},\ldots,p_{a_r})$.
	We study two questions.
	First, when do $p_{a_1},\ldots,p_{a_r}$ form a regular sequence?
	This holds if and only if $I$ has height $r$, that is, if and only if the common zero set of these power sums
	has dimension $n-r$.
	Second, when is $I$ a prime ideal?
	If $I$ is not prime, we want to find its minimal primes; they correspond to the irreducible components of
	the zero set.\footnote{This article is a substantially revised and extended version of unpublished work from the author's Ph.D. thesis \cite{KumarThesis} and of the preprint arXiv:1309.1098.}

	When the sequence is regular, the ring $S/I$ is Cohen--Macaulay; it is called a complete intersection.
	Its minimal graded free resolution is the Koszul complex.
	So the degrees $a_1,\ldots,a_r$ determine the graded Betti numbers and the Hilbert series of $S/I$;
	see \cite[Corollary~17.5]{DE}.
	The answer to the first question depends on the polynomials, and not only on their degrees:
	symmetric polynomials of degrees $a_1,\ldots,a_r$ may form a regular sequence although the power sums of the
	same degrees do not.
	For example, let $n=4$, and let $e_1,e_2,e_3$ be the elementary symmetric polynomials of degrees $1,2,3$.
	The symmetric polynomials $e_1^2$, $e_3^2$, $e_2^5$ have degrees $2$, $6$, $10$, and they form a regular
	sequence, because they have the same common zeros as $e_1,e_2,e_3$.
	But the power sums $p_2,p_6,p_{10}$ do not form a regular sequence; see Section~\ref{sec:quartic-pair}.

	The problem goes back to Conca, Krattenthaler and Watanabe \cite{CKW}.
	They started from a question in combinatorics.
	For $s\geq2$, the Hilbert series of the ring of symmetric invariants of $S/(x_1^{s},\ldots,x_n^{s})$ is a
	Gaussian binomial coefficient.
	They found a presentation of this ring of invariants by showing that $p_{s},p_{s+1},\ldots,p_{s+n-1}$ form a
	regular sequence \cite[Proposition~2.9]{CKW}.
	Then they asked which other sets of $n$ degrees give a regular sequence of power sums in $n$ variables
	\cite[Question~2.1]{CKW}.
	If the symmetric polynomials can be chosen freely, Galetto, Geramita and Wehlau \cite[Theorem~2.5]{GGW}
	list all sequences of degrees of regular sequences of $n$ homogeneous symmetric polynomials in $n\leq4$ variables.

	Kumar and Martino \cite{KM} studied power sums in three and four variables with Newton's identities and
	with prime ideals.
	They proved that at most $n-2$ consecutive power sums in $n\geq4$ variables generate a prime ideal
	\cite[Theorem~4.3]{KM}.
	Prime ideals help with the first question.
	Let $f_1,\ldots,f_r$ be a regular sequence of homogeneous polynomials, and suppose that the ideal
	$(f_1,\ldots,f_r)$ is prime.
	Then a homogeneous polynomial $f$ of positive degree can be added to the sequence if and only if
	$f\notin(f_1,\ldots,f_r)$.
	Conca, Singh and Soundararajan \cite[Theorem~2.3]{CSS} proved that $r$ power sums of distinct degrees always
	form a regular sequence if $n\geq2r-1$, and that $S/I$ is a normal domain if $n\geq2r+1$.
	For two power sums in four variables, they determined when $S/(p_a,p_b)$ is a normal domain
	\cite[Theorem~3.1]{CSS}.
	For three power sums, the bound $n\geq2r-1$ says that $n\geq5$.
	So, for three power sums, the answer to the first question can depend on the degrees only if $n\leq4$.

	For three power sums in three variables, the problem is open.
	Let $1\leq a<b<c$ be integers with $\gcd(a,b,c)=1$.
	If $p_a,p_b,p_c$ form a regular sequence in $\C[x_1,x_2,x_3]$, then $6$ divides $abc$ \cite[Lemma~2.8]{CKW}.
	Conca, Krattenthaler and Watanabe conjecture that the converse is also true \cite[Conjecture~2.10]{CKW}.
	Conca, Singh and Soundararajan proved this conjecture for $a=1$.
	They also proved that, for each fixed $a\geq2$, the conjecture has at most finitely many exceptions
	\cite[Theorems~4.1 and~5.1]{CSS}.
	More infinite families of regular sequences are in \cite[Theorem~1.1]{Abedelfatah}.
	In $n$ variables, Chen \cite[Theorem~2.4]{Chen} proved that $p_a,p_{a+d},\ldots,p_{a+(n-1)d}$ is a regular
	sequence if $\gcd(a,d)=1$ and $n!$ divides the product $a(a+d)\cdots(a+(n-1)d)$.
	In this paper the number of power sums is smaller than the number of variables.
	The main case is three power sums in four variables.

	Power sums also occur in other problems.
	Mel\'anov\'a, Sturmfels and Winter \cite{MSW} study how to recover $n$ numbers from the values of some of
	their power sums. For this they study the fibers and the images of the maps that are given by power sums.
	When the number of equations is equal to the number of unknowns, they relate regular sequences of power sums to
	solutions at infinity and to the B\'ezout bound.
	Hickman and Wright \cite[\S4]{HW} count the factorizations of $X^n$ into linear factors modulo prime powers.
	Their proof uses that the system of the power sums $p_1,\ldots,p_n$ is nondegenerate over the finite fields
	with $p$ elements, for all primes $p>n$; for one step they refer to \cite[Theorem~4.3]{KM}.
	These applications are one more reason to study the dimension and the singular points of the zero sets of
	power sums.

	\subsection*{Main results}
	We now state our main results. The first result is about power sums whose degrees are in arithmetic progression.
	It holds for every number of variables $n\geq3$.

	\begin{mainthm}\label{thm:A}
		Let $d\geq1$ and $1\leq m\leq n-2$ be integers.
		\begin{enumerate}[\rm(a)]
			\item For every integer $s\geq1$, the ring
			\[
			S/(p_{sd},p_{(s+1)d},\ldots,p_{(s+m-1)d})
			\]
			is a normal domain. In particular, $(p_d,p_{2d},\ldots,p_{md})$ is a prime ideal.
			\item Let $c\geq1$ be an integer that is not one of the degrees $d,2d,\ldots,md$.
			Then $p_d,p_{2d},\ldots,p_{md},p_c$ form a regular sequence in $S$ if and only if
			either $d$ does not divide $c$, or $d$ divides $c$ and $c/d$ is a sum of integers from the set
			$\{m+1,m+2,\ldots,n\}$. In such a sum, the same integer may be used several times.
		\end{enumerate}
	\end{mainthm}

	Part~(a) is Theorem~\ref{thm:power-sum-progressions}, and part~(b) is Theorem~\ref{thm:initial-semigroup}.
	The sums in part~(b) form the numerical semigroup generated by the integers $m+1,\ldots,n$; we call it
	$\Gamma_{m,n}$.
	A simple test to decide whether an integer lies in $\Gamma_{m,n}$, and the largest integer that does not lie in
	$\Gamma_{m,n}$, are known \cite{GSR,Roberts}; see Remark~\ref{rem:semigroup-interval}.
	For example, let $(m,n)=(3,5)$ and $c\notin\{d,2d,3d\}$. Then $p_d,p_{2d},p_{3d},p_c$ form a regular sequence if
	and only if $c\notin\{6d,7d,11d\}$.
	If $n\geq2m+1$, there is no exception at all. This case is contained in \cite[Theorem~2.3(1)]{CSS}.
	Theorem~\ref{thm:A}(b) also covers the range $m+2\leq n\leq2m$, where exceptions occur.

	Some cases of part~(a) are known.
	For $d=1$, that is, for consecutive degrees, the ideal is prime for $n\geq4$ by \cite[Theorem~4.3]{KM},
	and the origin is the only possible singular point by \cite[Theorem~2.6]{CSS}.
	For every $d$, the ring is a normal domain if $n\geq2m+1$, by \cite[Theorem~2.3(2)]{CSS}.
	So the new cases of part~(a) are those with $d\geq2$ and $m+2\leq n\leq2m$.
	We deduce the general case from \cite[Theorem~2.6]{CSS} with the substitution $x_i\mapsto x_i^d$.
	Part~(a) also proves the statement of the earlier preprint that $(p_d,p_{2d},\ldots,p_{md})$ is a prime ideal.
	Since this ideal is prime, part~(b) asks whether $p_c$ lies in it.
	If $d$ does not divide $c$, we show that it does not; for this we multiply one variable by a primitive $d$th
	root of unity.
	If $d$ divides $c$, we compute $p_c$ modulo the ideal with the formula of Girard and Waring, which expresses
	$p_c$ by the elementary symmetric polynomials.
	The same computation of coefficients appears in the proof of \cite[Theorem~2.11]{CKW}.

	Our other results are about four variables.
	So, from now on in this Introduction, $S=\C[x_1,x_2,x_3,x_4]$.
	The second result is about the ideal generated by two power sums $p_d$ and $p_{sd}$, where one degree divides
	the other.

	\begin{mainthm}\label{thm:B}
		Let $S=\C[x_1,x_2,x_3,x_4]$, and let $d\geq1$ and $s\geq2$ be integers.
		\begin{enumerate}[\rm(a)]
			\item The ideal $(p_d,p_{sd})$ is prime if and only if $s$ is even and, in addition, $s\ne4$ or $d$ is odd.
			\item Let $s=4$, and let $d$ be even. Put
			\[
			E_2=\sum_{1\leq i<j\leq4}x_i^dx_j^d,\qquad G=(x_1x_2x_3x_4)^{d/2}.
			\]
			Then the ideals $Q_+=(p_d,E_2+\sqrt2\,G)$ and $Q_-=(p_d,E_2-\sqrt2\,G)$ are prime, they are the minimal primes
			of $(p_d,p_{4d})$, and $(p_d,p_{4d})=Q_+\cap Q_-$.
			\item The ideal $(p_d,p_{3d})$ is the intersection of $3d^2$ prime ideals, and each of them is generated by
			two linear forms.
			The ideal $(p_d,p_{5d})$ is the intersection of these $3d^2$ prime ideals and of the prime ideal $(p_d,p_{2d})$.
		\end{enumerate}
	\end{mainthm}

	Part~(a) is Corollary~\ref{cor:pair-d-sd-prime}, and part~(b) is Theorem~\ref{thm:quartic-pair-components}(b).
	Part~(c) is contained in Proposition~\ref{thm:triple-d-3d}(a) and Corollary~\ref{cor:triple-d-5d}(a), where the
	prime ideals are written down.

	Some cases of Theorem~\ref{thm:B}(a) are known.
	For odd $s$, the ideal $(p_d,p_{sd})$ is not prime by \cite[Remark~3.3(1) and~(2)]{CSS}.
	For even $d$, the ideal $(p_d,p_{4d})$ is not prime by \cite[Remark~3.3(2) and~(6)]{CSS}.
	If $s$ is even and $s\not\equiv4\pmod6$, then $S/(p_d,p_{sd})$ is a normal domain by \cite[Theorem~3.1]{CSS};
	so $(p_d,p_{sd})$ is a prime ideal.
	This is also true for $d=1$ and every even $s$: the ideal $(p_1,p_s)$ was shown to be prime in
	\cite[Proposition~4.3]{KM}, and $S/(p_1,p_s)$ is a normal domain by \cite[Theorem~3.1]{CSS}.
	So the new cases of part~(a) are those with $d\geq2$ and $s\equiv4\pmod6$, that is, $s=6k+4$ with $k\geq0$,
	except the case $s=4$ with $d$ even.
	In these cases the ring $S/(p_d,p_{sd})$ is not normal.
	So we cannot prove that the ideal is prime by proving that the ring is normal, and we need another method.
	In part~(b), the element $E_2-\sqrt2\,G$ is known for $d=2$: modulo $p_2$, it is the element that is shown to
	be a zero-divisor of $S/(p_2,p_8)$ in \cite[Remark~3.3(6)]{CSS}.
	The new result in part~(b) is that $Q_+$ and $Q_-$ are prime and are the only minimal primes, for every
	even $d$.
	The generators of $Q_+$ and $Q_-$ contain $\sqrt2$. So the result depends on the field:
	over the rational numbers, the ideal $(p_2,p_8)$ is prime \cite[Remark~3.3(6)]{CSS}.

	Theorem~\ref{thm:B} has the following consequence.

	\begin{maincor}\label{cor:C}
		Let $S=\C[x_1,x_2,x_3,x_4]$, and let $b>3$ be an integer.
		The ring $S/(p_3,p_b)$ is a domain that is not normal if and only if $b=18k+12$ for some integer $k\geq0$.
	\end{maincor}

	This statement was conjectured by Conca, Singh and Soundararajan \cite[Remark~3.3(7)]{CSS}.
	We prove it in Corollary~\ref{cor:p3-domain}(a).
	Part~(b) of that corollary proves one part of their conjecture on the rings $S/(p_2,p_b)$
	\cite[Remark~3.3(4)]{CSS}: the ring $S/(p_2,p_{12k+8})$ is a domain that is not normal for every
	integer $k\geq1$.

	The proof of Theorem~\ref{thm:B}(a) and~(b) uses the curve $C$ in the projective space $\mathbb P^3$ that is
	defined by $p_1=p_s=0$.
	The zero set of $(p_d,p_{sd})$ in $\mathbb P^3$ is the inverse image of $C$ under the map that raises each
	coordinate to its $d$th power, and its irreducible components correspond to the minimal primes of $(p_d,p_{sd})$.
	Let $y_1,\ldots,y_4$ be the coordinates of the space that contains $C$.
	With Kummer theory we reduce the count of the components to a question about the functions $y_1/y_4$, $y_2/y_4$,
	$y_3/y_4$ on $C$: which products of powers of these three functions are $d$th powers?
	The answer depends on the orders of the zeros of these functions, and these orders are the multiplicities of
	the roots of the polynomial $1+X^s+(1+X)^s$.
	If $s\ne4$, this polynomial has a simple root, and we get one component.
	If $s=4$, every root is a double root. In this case $C$ is a smooth plane quartic, and the product of the three
	functions is a square.
	Then we get one component for odd $d$, and two components for even $d$.

	The third result is about three power sums in four variables.
	For a positive integer $a$, let $\nu_2(a)$ be the largest integer $v\geq0$ such that $2^v$ divides $a$.

	\begin{mainthm}\label{thm:D}
		Let $S=\C[x_1,x_2,x_3,x_4]$, and let $d\geq1$ and $c\geq1$ be integers.
		\begin{enumerate}[\rm(a)]
			\item Let $c\notin\{d,2d\}$. Then $p_d,p_{2d},p_c$ form a regular sequence in $S$ if and only if $c\ne5d$.
			\item Let $c\notin\{d,3d\}$. Then $p_d,p_{3d},p_c$ form a regular sequence in $S$ if and only if
			$\nu_2(d)\ne\nu_2(c)$.
			\item Let $c\notin\{d,4d\}$. Then $p_d,p_{4d},p_c$ form a regular sequence in $S$.
			\item Let $c\notin\{d,5d\}$. Then $p_d,p_{5d},p_c$ form a regular sequence in $S$ if and only if
			$c\ne2d$ and $\nu_2(d)\ne\nu_2(c)$.
			\item Let $s\geq2$ be an even integer, and let $c\notin\{d,sd\}$. If $d$ does not divide $c$, or if $c/d$ is an
			even integer, then $p_d,p_{sd},p_c$ form a regular sequence in $S$.
		\end{enumerate}
	\end{mainthm}

	The five parts of Theorem~\ref{thm:D} are Corollary~\ref{thm:triple-d-2d}, Proposition~\ref{thm:triple-d-3d}(c),
	Theorem~\ref{thm:triple-d-4d}, Corollary~\ref{cor:triple-d-5d}(b) and Corollary~\ref{rem:quartic-triple-context}.
	In all parts, $c$ need not be a multiple of $d$, and $c$ need not be the largest of the three degrees.
	The order of the three power sums does not matter.
	Part~(a) is the case $(m,n)=(2,4)$ of Theorem~\ref{thm:A}(b).
	The exception $c=5d$ is recorded in \cite[Remark~2.16(c)]{CKW}; part~(a) shows that there is no other exception.
	For $c>3d$, part~(b) follows from results on three variables in \cite[Proposition~2.13 and Lemma~2.4(a)]{CKW}.
	We give a direct proof in four variables.

	In the proof of Theorem~\ref{thm:D} we use the minimal primes of Theorem~\ref{thm:B}:
	the power sums $p_d,p_{sd},p_c$ form a regular sequence if and only if $p_c$ lies in no minimal prime of $(p_d,p_{sd})$.
	For example, let $d$ be even. Then $S/(p_d,p_{4d})$ is not a domain.
	So it is not enough to show that $p_c\notin(p_d,p_{4d})$: we must show that $p_c$ lies neither in $Q_+$ nor in $Q_-$.
	In geometric terms, two power sums of distinct degrees in four variables define a curve in $\mathbb P^3$
	\cite[Theorem~2.3(1) and~(4)]{CSS}.
	A third power sum $p_c$ can be added to them if and only if the surface $p_c=0$ contains no irreducible
	component of this curve.
	In this case the three power sums have only finitely many common zeros in $\mathbb P^3$.
	Counted with multiplicities, the number of these zeros is the product of the three degrees; for the degrees $d$, $4d$, $c$
	it is $4d^2c$.

	Theorem~\ref{thm:D} suggests a question about all triples of degrees;
	this is Question~\ref{question:triple-normalized-classification}.
	Let $1\leq a<b<c$ be integers, and put $g=\gcd(a,b,c)$.
	Do $p_a,p_b,p_c$ form a regular sequence in $S$ if and only if the integers $a/g$, $b/g$, $c/g$ are not all odd and
	$(a,b,c)$ is not of the form $(d,2d,5d)$?
	Both conditions are necessary.
	By Theorem~\ref{thm:D}, the answer is yes for every triple that contains $d$ and $2d$, or $d$ and $3d$, or
	$d$ and $4d$, or $d$ and $5d$, for some integer $d\geq1$.
	The two conditions correct \cite[Conjecture~4.5]{KM}.
	We do not know the answer in general.

	All statements about regular sequences in this paper are valid over every field of characteristic zero;
	see Section~\ref{sec:preliminaries}.
	The statements about prime ideals depend on the field, and we prove them over $\C$.

	Our computations were done with CoCoA-5 \cite{CoCoA} and repeated independently with Macaulay2 \cite{M2}.

	\section{Preliminaries}\label{sec:preliminaries}

In this section we fix the notation. We also collect the known results that we use in more than one section.
These are Newton's identities, facts on regular sequences, a lemma on the substitution $x_i\mapsto x_i^d$,
the Jacobian criterion, and a theorem of Conca, Singh and Soundararajan on ideals generated by power sums.
For completeness, we add short explanations or proofs.
Results that we use in only one section are recalled in that section.

	\subsection*{Notation and Newton's identities}

	\begin{remark}[Notation]\label{rem:notation}
		\begin{enumerate}[\rm(a)]
			\item Throughout, $n\geq1$ is an integer and
			$S=\C[x_1,\ldots,x_n]$, with $\deg x_i=1$ for every $i$.
			We write $\m=(x_1,\ldots,x_n)$ for its homogeneous maximal ideal.
			We use $\N=\{1,2,\ldots\}$ and $\N_0=\N\cup\{0\}$.
			For an ideal $I\subseteq S$, let $V(I)$ be the set of common zeros
			of its elements in $\C^n$.
			If $I\ne S$, then $\hgt I+\dim S/I=n$, and $\dim S/I$ is the dimension of $V(I)$.

			\item For $k\geq1$, let $p_k$ and $e_k$ denote the power sum and the elementary symmetric polynomial
			of degree $k$, respectively. Thus
			\[
			p_k=\sum_{i=1}^n x_i^k,\qquad
			e_k=\sum_{1\leq i_1<\cdots<i_k\leq n}x_{i_1}\cdots x_{i_k}.
			\]
			We set $e_0=1$, and $e_k=0$ for $k>n$.
			When the variables need to be specified, we write
			$p_k(x_1,\ldots,x_n)$, and similarly for $e_k$.

			\item The letter $d$ always denotes a positive integer.
			In Sections~\ref{sec:power-sum-progressions}--\ref{sec:quartic-pair} we compare the power sums of degrees
			$d,2d,3d,\ldots$ with a power sum $p_c$ of another degree $c$, and we write $q=c/d$ when $d$ divides $c$.

			\item For a prime number $\ell$ and a positive integer $a$, let
			$\nu_\ell(a)$ denote the largest integer $v\geq0$ such that
			$\ell^v$ divides $a$. For example, $\nu_2(12)=2$ and $\nu_3(12)=1$.
		\end{enumerate}
	\end{remark}

	\begin{citedremark}[{\cite[Chapter~I, \S2, (2.11$'$)]{I-G}}]
		\label{rem:newton-identities}
		\begin{enumerate}[\rm(a)]
			\item For every $k\geq1$, Newton's identities give
			\begin{equation}\label{eq:newton-identities}
				ke_k=\sum_{j=1}^k(-1)^{j-1}e_{k-j}p_j.
			\end{equation}
			These identities allow us to pass between power sums and elementary symmetric polynomials.

			\item In particular, for $1\leq r\leq n$ we have the following equality of ideals of $S$:
			\[
			(p_1,\ldots,p_r)=(e_1,\ldots,e_r).
			\]
			This is well known, and it follows from \eqref{eq:newton-identities} by induction on $r$.
			For $r=1$ the two ideals are equal, because $p_1=e_1$.
			Let $r\geq2$, and let $I$ be the ideal $(p_1,\ldots,p_{r-1})=(e_1,\ldots,e_{r-1})$.
			In \eqref{eq:newton-identities} with $k=r$, every term with $j<r$ contains the factor $p_j$, which lies in $I$.
			Hence
			\[
			re_r\equiv(-1)^{r-1}p_r\pmod I.
			\]
			Since $r$ is invertible in $\C$, adjoining $p_r$ or $e_r$ to $I$ gives the same ideal.
		\end{enumerate}
	\end{citedremark}

	\subsection*{Regular sequences and the substitution \texorpdfstring{$x_i\mapsto x_i^d$}{x-i to x-i-d}}

	We recall the definition of a regular sequence and the height criterion.
	Then we study the substitution $x_i\mapsto x_i^d$.
	It replaces the power sum $p_k$ by $p_{kd}$, and it does not change whether a sequence is regular.

	\begin{citedremark}[{\cite[Proposition~18.13 and Corollary~18.14]{DE}}]
		\label{rem:regular-sequences}
		\begin{enumerate}[\rm(a)]
			\item A sequence $f_1,\ldots,f_r$ in a ring $A$ is called an
			$A$-regular sequence if $(f_1,\ldots,f_r)\ne A$ and, for every $i$,
			the image of $f_i$ is a nonzerodivisor on $A/(f_1,\ldots,f_{i-1})$.
			For $i=1$, the preceding ideal is $(0)$, so $f_1$ must be a
			nonzerodivisor on $A$.
			An $S$-regular sequence is simply called a regular sequence.

			\item Let $f_1,\ldots,f_r$ be homogeneous elements of positive degrees
			in $S$. Since $S$ is Cohen--Macaulay, these elements form a regular
			sequence if and only if
			\begin{equation}\label{eq:height-criterion}
				\hgt(f_1,\ldots,f_r)=r.
			\end{equation}
			We call this the height criterion.
			In this case, $S/(f_1,\ldots,f_r)$ is a Cohen--Macaulay ring of dimension $n-r$; it is called a complete intersection.
			Put $I=(f_1,\ldots,f_r)$. The associated primes of $I$ are exactly its minimal primes,
			and each of them has height $r$.
			So $I$ has no embedded primes, and every irreducible component of $V(I)$ has dimension $n-r$.
			Moreover, let $\mathfrak p$ be a prime ideal of $S$ that contains $I$.
			Then the prime ideal $\mathfrak p/I$ of $S/I$ has height $\hgt\mathfrak p-r$,
			because every minimal prime of $I$ has height $r$.
			Since $\hgt\mathfrak p+\dim S/\mathfrak p=n$, we get
			\[
			\hgt(\mathfrak p/I)+\dim S/\mathfrak p=n-r.
			\]

			\item In particular, $n$ homogeneous elements $f_1,\ldots,f_n$ of positive degrees
			form a regular sequence in $S$ if and only if their only common zero
			is the origin. To see this, put $I=(f_1,\ldots,f_n)$.
			If $\hgt I=n$, then $\dim S/I=0$, so $V(I)$ is a finite set. A nonzero common
			zero $w$ would give the line $\{\lambda w:\lambda\in\C\}$ of
			common zeros, because the $f_i$ are homogeneous. This is impossible,
			because a line has infinitely many points.
			Conversely, if $V(I)=\{0\}$, then the Nullstellensatz gives
			$\sqrt I=\m$, so $\hgt I=n$.

			\item We use two more consequences of (b).
			Let $f_1,\ldots,f_r$ be a regular sequence of homogeneous elements of positive degrees in $S$.
			First, condition \eqref{eq:height-criterion} does not depend on the order of the elements.
			So the sequence remains a regular sequence after any permutation.
			Also, if we remove some elements from the sequence, the remaining elements still form a
			regular sequence: after a permutation, they are the first elements of the sequence.
			Second, let $f$ be one more homogeneous element of positive degree.
			An element of $S$ is a zerodivisor on $S/(f_1,\ldots,f_r)$ exactly when it lies in an associated prime
			of $(f_1,\ldots,f_r)$, and by (b) these are the minimal primes.
			Hence $f_1,\ldots,f_r,f$ is a regular sequence if and only if $f$ lies in no minimal prime of
			$(f_1,\ldots,f_r)$. If $(f_1,\ldots,f_r)$ is a prime ideal, this means that $f\notin(f_1,\ldots,f_r)$.
		\end{enumerate}
	\end{citedremark}

	\begin{citedlemma}[{\cite[Corollary~6.3 and Exercise~6.7]{DE}}]
		\label{lem:finite-free-substitution}
		Let $d\geq1$ be an integer, and let $T=\C[y_1,\ldots,y_n]$. Consider the homomorphism of $\C$-algebras
		\[
		T\longrightarrow S,\qquad y_i\longmapsto x_i^d,
		\]
		which sends a polynomial $g(y_1,\ldots,y_n)$ to $g(x_1^d,\ldots,x_n^d)$.
		\begin{enumerate}[\rm(a)]
			\item This homomorphism is injective. If we identify $T$ with its image in $S$, then
			\begin{equation}\label{eq:substitution-basis}
				S=\bigoplus_{0\leq u_1,\ldots,u_n<d}
				T x_1^{u_1}\cdots x_n^{u_n}.
			\end{equation}
			So $S$ is a free $T$-module of rank $d^n$.
			\item Let $g_1,\ldots,g_r\in T$. Then $g_1,\ldots,g_r$ is a regular sequence in $T$ if and only if
			\[
			g_1(x_1^d,\ldots,x_n^d),\ldots,g_r(x_1^d,\ldots,x_n^d)
			\]
			is a regular sequence in $S$.
		\end{enumerate}
	\end{citedlemma}

	\begin{proof}
		(a) The homomorphism is injective, since distinct monomials in the $y_i$ have distinct
		images. We identify $T$ with its image in $S$; this image is the subring $\C[x_1^d,\ldots,x_n^d]$.
		Every exponent of $x_i$ can be written uniquely as $dv_i+u_i$,
		where $v_i\geq0$ and $0\leq u_i<d$ are integers.
		So every monomial of $S$ is, in exactly one way, the product of a monomial in
		$x_1^d,\ldots,x_n^d$ and a monomial $x_1^{u_1}\cdots x_n^{u_n}$ with $0\leq u_i<d$ for all $i$.
		Grouping monomials according
		to these remainders gives
		\eqref{eq:substitution-basis}.
		Thus $S$ is free of rank $d^n$ over $T$.

		(b) For an ideal $\mathfrak a\subseteq T$, the decomposition \eqref{eq:substitution-basis} shows that
		$S/\mathfrak aS$ is a direct sum of $d^n$ copies of the $T$-module $T/\mathfrak a$.
		So $\mathfrak aS\ne S$ if and only if $\mathfrak a\ne T$, and an element of $T$ is a nonzerodivisor on
		$S/\mathfrak aS$ if and only if it is a nonzerodivisor on $T/\mathfrak a$.
		We use this for the ideals $\mathfrak a=(g_1,\ldots,g_{i-1})$ with $1\leq i\leq r$, and for
		$\mathfrak a=(g_1,\ldots,g_r)$. By the definition of a regular sequence in
		Remark~\ref{rem:regular-sequences}(a), this proves (b).
	\end{proof}

	We often use Lemma~\ref{lem:finite-free-substitution}(b) in the following way.
	Take polynomials $f_1,\ldots,f_r\in S$, and replace every variable $x_i$ by $x_i^d$.
	We call this operation the \emph{substitution} $x_i\mapsto x_i^d$.
	The new polynomials $f_1(x_1^d,\ldots,x_n^d),\ldots,f_r(x_1^d,\ldots,x_n^d)$ form a regular sequence
	in $S$ if and only if $f_1,\ldots,f_r$ form a regular sequence in $S$.
	To see this, apply Lemma~\ref{lem:finite-free-substitution}(b) to the polynomials
	$g_j=f_j(y_1,\ldots,y_n)$ of $T$.
	For power sums, the substitution gives
	$p_j(x_1^d,\ldots,x_n^d)=p_{jd}(x_1,\ldots,x_n)$.
	So, for positive integers $a_1,\ldots,a_r$, the power sums $p_{a_1},\ldots,p_{a_r}$ form a regular
	sequence in $S$ if and only if $p_{da_1},\ldots,p_{da_r}$ form a regular sequence in $S$.
	Conca, Krattenthaler and Watanabe \cite[Lemma~2.2]{CKW} state this equivalence for $n$ power sums in $n$ variables;
	see also \cite[Remark~1.2(2)]{CSS}. Conca, Singh and Soundararajan use the same map for two power sums
	\cite[Remark~3.3(2)]{CSS}.
	Lemma~\ref{lem:finite-free-substitution} gives it for sequences of any length,
	and for arbitrary polynomials in place of power sums.

	The same argument works when the field of coefficients is made larger; see also \cite[Remark~1.2(1)]{CSS}.
	Let $K\subseteq K'$ be fields. A basis of $K'$ as a vector space over $K$ is also a basis of
	$K'[x_1,\ldots,x_n]$ as a module over $K[x_1,\ldots,x_n]$.
	So, as in the proof of Lemma~\ref{lem:finite-free-substitution}(b), polynomials with coefficients in $K$
	form a regular sequence in $K[x_1,\ldots,x_n]$ if and only if they form a regular sequence in
	$K'[x_1,\ldots,x_n]$.
	Power sums have coefficients in $\mathbb Q$, and every field of characteristic zero contains $\mathbb Q$.
	Hence every condition for regular sequences of power sums that we prove over $\C$
	is valid over every field of characteristic zero.

	\subsection*{The Jacobian criterion}

	We recall the Jacobian criterion for a complete intersection.
	We use it in Sections~\ref{sec:power-sum-progressions} and~\ref{sec:quartic-pair}.

	\begin{citedremark}[{\cite[Theorem~16.19 and Corollary~16.20]{DE}}]
		\label{rem:jacobian-criterion}
		Let $f_1,\ldots,f_r$ be a regular sequence of homogeneous elements of positive degrees in $S$, with $1\leq r\leq n$. Set $I=(f_1,\ldots,f_r)$
		and $R=S/I$. Let
		\[
		M=\left(\frac{\partial f_i}{\partial x_j}\right)_{
			1\leq i\leq r,\,1\leq j\leq n},\qquad J=I_r(M),
		\]
		where $I_r(M)$ is the ideal of $S$ generated by the $r\times r$
		minors of $M$. The matrix $M$ is called the Jacobian matrix of $f_1,\ldots,f_r$.
		For a point $w\in\C^n$, we write $M(w)$ for the matrix of complex numbers that we get
		when we evaluate the entries of $M$ at $w$.

		The prime ideals of $R$ are the ideals $\mathfrak p/I$, where $\mathfrak p$ is a prime ideal
		of $S$ that contains $I$. We write $R_{\mathfrak p}$ for the local ring of $R$ at $\mathfrak p/I$.
		The Jacobian criterion says that
		\[
		R_{\mathfrak p}\text{ is regular}
		\quad\Longleftrightarrow\quad
		J\not\subseteq\mathfrak p.
		\]
		In words: $R_{\mathfrak p}$ is regular if and only if some $r\times r$ minor of $M$ does not
		lie in $\mathfrak p$.

		A point $w\in V(I)$ gives the maximal ideal $\m_w=(x_1-w_1,\ldots,x_n-w_n)$ of $S$,
		and $I\subseteq\m_w$. A polynomial lies in $\m_w$ exactly when it is zero at $w$.
		Taking $\mathfrak p=\m_w$, we see that the local ring of
		$R$ at the point $w$ is regular if and only if $\rank M(w)=r$.
		We call $w$ a singular point of $\Spec R$ if this local ring is not regular.
		Thus $w$ is singular exactly when every $r\times r$ minor of $M$ vanishes
		at $w$. The set of singular points is therefore $V(I+J)$.
		In the same way, the prime ideals $\mathfrak p/I$ of $R$ for which $R_{\mathfrak p}$ is not regular
		are exactly those with $I+J\subseteq\mathfrak p$. They form the singular locus of $\Spec R$.
		The singular locus is empty if $I+J=S$. Otherwise its dimension is $\dim S/(I+J)$.
	\end{citedremark}

	\subsection*{Ideals generated by power sums}

	We recall a theorem of Conca, Singh and Soundararajan on ideals generated by power sums of distinct degrees.

	\begin{citedtheorem}[{\cite[Theorem~2.3(1), (2) and~(4)]{CSS}}]
		\label{rem:css-general}
		Let $a_1,\ldots,a_r$ be distinct positive integers, and put
		$I=(p_{a_1},\ldots,p_{a_r})\subseteq S$.
		\begin{enumerate}[\rm(a)]
			\item If $n\geq2r-1$, then $I$ is a complete intersection of
			height $r$; that is, $p_{a_1},\ldots,p_{a_r}$ is a regular sequence.
			\item If $n\geq2r+1$, then $S/I$ is a normal domain.
			\item If $n\geq2r$, then $S/I$ is reduced.
		\end{enumerate}
	\end{citedtheorem}

	We use this theorem mostly for two power sums.
	By (a), two power sums of distinct degrees form a regular sequence when $n\geq3$.
	In four variables, the ideal generated by any two power sums of
	distinct degrees is a radical complete intersection of
	height two, by (a) and (c).
	So its zero set has dimension two, and the ideal is the intersection of its minimal primes.
	We use both facts later.

	\section{Power sums in arithmetic progression}
	\label{sec:power-sum-progressions}

	For integers $d,s\geq1$ and $1\leq m\leq n-2$, we prove that
	\[
	S/(p_{sd},p_{(s+1)d},\ldots,p_{(s+m-1)d})
	\]
	is a normal domain. We then use this fact to determine, for $d\geq1$,
	exactly which integers $c$ outside $\{d,2d,\ldots,md\}$ make $p_d,p_{2d},\ldots,p_{md},p_c$
	a regular sequence in $S$.
	If $d$ does not divide $c$, the sequence is always regular.
	If $c=dq$, the sequence is regular exactly when $q$ belongs to the semigroup generated by
	$m+1,\ldots,n$.
	Chen \cite[Theorem~2.4]{Chen} studied another kind of arithmetic progression: the $n$ power sums
	$p_a,p_{a+d},\ldots,p_{a+(n-1)d}$ in $n$ variables, where $\gcd(a,d)=1$.
	In this section the common difference $d$ divides every degree, and there are at most $n-2$ power sums.

	The case of consecutive degrees is known. To pass to the degrees $sd,(s+1)d,\ldots$,
	we look at the nonzero coordinates of a point, and we use the substitution $x_i\mapsto x_i^d$.
	We first recall the known results on consecutive power sums, and a criterion for normal domains.

	\subsection*{Consecutive power sums and normal domains}

	\begin{citedprop}[{\cite[Proposition~2.9]{CKW}}]
		\label{lem:consecutive-sequences}
		For every integer $s\geq1$, the consecutive power sums
		\[
		p_s,p_{s+1},\ldots,p_{s+n-1}
		\]
		form a regular sequence in $S$.
		More generally, for all integers $d,s\geq1$, the sequence
		\begin{equation}\label{eq:progression-regular}
			p_{sd},p_{(s+1)d},\ldots,p_{(s+n-1)d}
		\end{equation}
		is a regular sequence in $S$.
		So are its first $m$ elements, for every $m$ with $1\leq m\leq n$.
	\end{citedprop}

	\begin{proof}
		The first statement is \cite[Proposition~2.9]{CKW}.
		The sequence \eqref{eq:progression-regular} is obtained from $p_s,p_{s+1},\ldots,p_{s+n-1}$
		by the substitution $x_i\mapsto x_i^d$.
		So it is a regular sequence by Lemma~\ref{lem:finite-free-substitution}(b).
		The last statement follows from the definition in Remark~\ref{rem:regular-sequences}(a):
		the first $m$ elements of a regular sequence form a regular sequence.
	\end{proof}

	\begin{citedtheorem}[{\cite[Theorem~2.6]{CSS}}]
		\label{rem:css-consecutive}
		Let $s$ and $r$ be positive integers. The origin is the only possible
		singular point of
		\[
		\Spec\bigl(S/(p_s,p_{s+1},\ldots,p_{s+r-1})\bigr).
		\]
	\end{citedtheorem}

	We say ``only possible'' because the singular locus can be empty. For example, $S/(p_1)$ is a
	polynomial ring in $n-1$ variables over $\C$, so its spectrum
	has no singular points.

	Next we recall Serre's criterion \cite[Theorem~11.5]{DE}.
	A Noetherian ring $A$ is called normal if $A_{\mathfrak p}$ is an
	integrally closed domain for every prime ideal $\mathfrak p$.
	A normal domain is a normal ring that is a domain.
	By Serre's criterion, $A$ is normal if and only if it satisfies
	both of the following conditions:
	\begin{enumerate}[\rm(a)]
		\item $(R_1)$: the local ring $A_{\mathfrak p}$ is regular whenever
		$\hgt\mathfrak p\leq1$;
		\item $(S_2)$: for every prime ideal $\mathfrak p$,
		\[
		\depth A_{\mathfrak p}\geq\min\{2,\dim A_{\mathfrak p}\}.
		\]
	\end{enumerate}
	A Cohen--Macaulay ring satisfies $(S_2)$, because
	$\depth A_{\mathfrak p}=\dim A_{\mathfrak p}$ for all its prime ideals $\mathfrak p$.
	So a Cohen--Macaulay ring is normal if and only if it satisfies $(R_1)$.
	A Noetherian normal ring is a finite product of normal domains \cite[Proposition~2.20]{DE}.
	An element $u$ of a ring is called an idempotent if $u^2=u$.
	A product of two nonzero rings contains the idempotent $(1,0)$, which is different from $0$ and $1$.
	So a normal ring $A\ne0$ in which $0$ and $1$ are the only idempotents is a normal domain.

	With Serre's criterion and the Jacobian criterion in Remark~\ref{rem:jacobian-criterion}, one can prove that a
	quotient $S/I$ by a homogeneous ideal $I$ is a normal domain: it is enough to show that its singular locus is
	small. The exact statement is the following proposition.

	\begin{citedprop}[{\cite[Theorem~18.15]{DE}}]
		\label{prop:normality-criterion}
		With the notation of Remark~\ref{rem:jacobian-criterion}, suppose that
		$I+J=S$ or that
		\begin{equation}\label{eq:singular-dimension-bound}
			\dim S/(I+J)\leq n-r-2.
		\end{equation}
		Then $S/I$ is a Cohen--Macaulay normal domain. In particular, $I$ is prime.
		This applies, for example, if $n-r\geq2$ and the origin is the only possible
		singular point of $\Spec(S/I)$.
	\end{citedprop}

	The result in \cite{DE} gives a product of normal domains.
	We include a proof, because we need that $S/I$ is a domain.
	This holds because the ideal $I$ is homogeneous.
	The same argument is used in the proof of \cite[Theorem~4.3]{KM}.

	\begin{proof}
		Put $R=S/I$. By Remark~\ref{rem:regular-sequences}(b), the ring $R$ is Cohen--Macaulay.
		Hence $R$ satisfies $(S_2)$.

		We show that $R$ satisfies $(R_1)$.
		Let $\mathfrak p$ be a prime ideal of $S$ that contains $I$, and suppose that the prime ideal
		$\mathfrak p/I$ of $R$ has height at most one.
		By the last formula in Remark~\ref{rem:regular-sequences}(b),
		\[
		\dim S/\mathfrak p=n-r-\hgt(\mathfrak p/I)\geq n-r-1.
		\]
		Suppose that $R_{\mathfrak p}$ is not regular.
		Then $J\subseteq\mathfrak p$ by the Jacobian criterion in Remark~\ref{rem:jacobian-criterion}.
		So $I+J\subseteq\mathfrak p$. In particular $I+J\ne S$, and \eqref{eq:singular-dimension-bound} gives
		\[
		\dim S/\mathfrak p\leq\dim S/(I+J)\leq n-r-2.
		\]
		This is a contradiction. So $R_{\mathfrak p}$ is regular, and $R$ satisfies $(R_1)$.
		Serre's criterion, recalled above, now shows that $R$ is normal.

		Hence $R$ is a finite product of normal domains.
		We show that there is only one factor. Since $I$ is homogeneous
		and its generators have positive degrees, $R$ is a graded ring $R=\bigoplus_{k\geq0}R_k$ with $R_0=\C$.

		Suppose that $u^2=u$ in $R$. The degree-zero component $u_0$ of $u$ satisfies
		$u_0^2=u_0$, so $u_0=0$ or $u_0=1$. Replacing $u$ by $1-u$ if
		necessary, we may assume that $u_0=0$. If $u\ne0$, let $j>0$ be
		the least degree in which $u$ has a nonzero component.
		Then every nonzero component of $u^2$ has degree at least $2j$.
		So the component of $u^2$ of degree $j$ is zero, and this contradicts $u^2=u$.
		Thus the only idempotents of $R$ are $0$ and $1$.
		A product of two nonzero rings contains an idempotent different from $0$ and $1$.
		So the product has only one factor, and $R$ is a domain.

		It remains to prove the last statement of the proposition.
		Suppose that $n-r\geq2$ and that no point of $V(I)$ other than the origin is singular.
		Then $V(I+J)\subseteq\{0\}$. If $V(I+J)$ is empty, then $I+J=S$ by the Nullstellensatz.
		Otherwise $V(I+J)=\{0\}$, and so $\dim S/(I+J)=0\leq n-r-2$.
	\end{proof}

	\subsection*{The ideals \texorpdfstring{$(p_{sd},p_{(s+1)d},\ldots,p_{(s+m-1)d})$}{(p-sd, ..., p-(s+m-1)d)}}

	The next theorem is the first main result of this section.
	Its proof uses Theorem~\ref{rem:css-consecutive} for the nonzero coordinates of a point,
	after the substitution $x_i\mapsto x_i^d$.

	\begin{samepage}
		\begin{Theorem}
			\label{thm:power-sum-progressions}
			Let $d,s\geq1$ and $1\leq m\leq n-2$ be integers, and put
			\[
			I=(p_{sd},p_{(s+1)d},\ldots,p_{(s+m-1)d})\subseteq S.
			\]
			Then $I$ is a complete intersection ideal of height $m$, and $S/I$
			is a Cohen--Macaulay normal domain of dimension $n-m$.
			Moreover, the origin is the only possible singular point of
			$\Spec(S/I)$.
			In particular, $I$ is prime.
		\end{Theorem}
	\end{samepage}

	Some cases of Theorem~\ref{thm:power-sum-progressions} are known.
	If $n\geq2m+1$, the theorem is contained in \cite[Theorem~2.3(2)]{CSS}, for every $d$;
	this includes the case $m=1$.
	For consecutive degrees, that is for $d=1$, and for $n\geq4$, the ideal $I$ was shown to be prime in
	\cite[Theorem~4.3]{KM}; the proof given there also shows that $S/I$ is normal.
	So, compared with these published results, the new cases are those with $d\geq2$ and $m+2\leq n\leq2m$.
	In the earlier version of this paper, arXiv:1309.1098, Theorem~3.13 stated that the ideal
	$(p_d,p_{2d},\ldots,p_{md})$ is prime for $1\leq m\leq n-2$; this is the case $s=1$.
	Theorem~\ref{thm:power-sum-progressions} proves more: it allows every $s\geq1$, and it shows that
	the quotient is a normal domain.

	\begin{proof}
		By Proposition~\ref{lem:consecutive-sequences}, the sequence
		\[
		p_{sd},p_{(s+1)d},\ldots,p_{(s+n-1)d}
		\]
		is a regular sequence in $S$. Its first $m$ elements generate $I$.
		They form a regular sequence, again by Proposition~\ref{lem:consecutive-sequences}. So $\hgt I=m$, and
		$R=S/I$ is a Cohen--Macaulay complete intersection of dimension $n-m$ by
		Remark~\ref{rem:regular-sequences}(b).

		We show that no point of $V(I)$ other than the origin is a singular point of $\Spec R$.
		Take a point $w=(w_1,\ldots,w_n)$ of $V(I)$ with $w\ne0$.
		The ideal $I$ does not change when we permute the variables. So, after permuting the coordinates, we may assume that exactly
		$w_1,\ldots,w_k$ are nonzero. Here $k\geq1$, because $w\ne0$. Put
		\[
		z=(w_1^d,\ldots,w_k^d)\in\C^k,
		\qquad T=\C[y_1,\ldots,y_k],
		\]
		and define
		\[
		F_i(y_1,\ldots,y_k)=\sum_{j=1}^k y_j^{s+i}
		\qquad(i\geq0).
		\]
		So $F_i$ is the power sum of degree $s+i$ in the $k$ variables $y_1,\ldots,y_k$.
		Below we use Proposition~\ref{lem:consecutive-sequences}, Theorem~\ref{rem:css-consecutive} and the results
		of Section~\ref{sec:preliminaries} for the ring $T$, with $k$ in place of $n$.
		The remaining coordinates of $w$ are zero. Hence
		\[
		F_i(z)=\sum_{j=1}^k w_j^{(s+i)d}
		=p_{(s+i)d}(w)=0
		\qquad(0\leq i\leq m-1).
		\]

		We first show that $k>m$.
		Suppose that $k\leq m$. Then $F_0,\ldots,F_{k-1}$ are among the polynomials
		$F_0,\ldots,F_{m-1}$, which are zero at $z$.
		So $z$ is a common zero of the $k$ consecutive power sums $F_0,\ldots,F_{k-1}$ in $k$ variables.
		By Proposition~\ref{lem:consecutive-sequences}, applied in $T$, these
		form a regular sequence. Their only common zero is the origin,
		by Remark~\ref{rem:regular-sequences}(c).
		This contradicts the fact that every coordinate of $z$ is nonzero.

		Since $k>m$, the polynomials $F_0,\ldots,F_{m-1}$ are the first $m$ of the $k$ consecutive power sums
		$F_0,\ldots,F_{k-1}$. By Proposition~\ref{lem:consecutive-sequences}, applied in $T$, they form a
		regular sequence in $T$. So $T/(F_0,\ldots,F_{m-1})$ is a complete intersection of dimension $k-m$.
		We apply Theorem~\ref{rem:css-consecutive} to these $m$ consecutive power sums in the
		$k$ variables $y_1,\ldots,y_k$.
		The point $z$ is not the origin. So $z$ is not a singular point of
		$\Spec\bigl(T/(F_0,\ldots,F_{m-1})\bigr)$.
		By the Jacobian criterion in Remark~\ref{rem:jacobian-criterion}, applied in $T$ to the regular sequence
		$F_0,\ldots,F_{m-1}$, the Jacobian matrix of these polynomials at $z$,
		\[
		M_F(z)=\left(\frac{\partial F_i}{\partial y_j}(z)\right)_
		{0\leq i\leq m-1,\,1\leq j\leq k}
		\]
		has rank $m$.

		Let $M(w)$ be the Jacobian matrix of the generators of $I$ at $w$,
		and let $M_+(w)$ consist of its first $k$ columns.
		Let $1\leq j\leq k$. The variable $x_j$ occurs in $p_{(s+i)d}$ only in the term
		$x_j^{(s+i)d}=(x_j^d)^{s+i}$. So the chain rule gives
		\[
		\frac{\partial p_{(s+i)d}}{\partial x_j}(w)
		=d w_j^{d-1}\frac{\partial F_i}{\partial y_j}(z).
		\]
		In matrix form, this says that
		\begin{equation}
			M_+(w)=M_F(z)\operatorname{diag}
			(d w_1^{d-1},\ldots,d w_k^{d-1}).
		\end{equation}
		The diagonal entries $dw_j^{d-1}$ are not zero, because $w_1,\ldots,w_k$ are not zero.
		So the diagonal matrix is invertible, and $M_+(w)$ has the same rank as $M_F(z)$, namely $m$.
		The matrix $M(w)$ has $m$ rows, and $M_+(w)$ is a part of it. So $M(w)$ also has rank $m$.
		By the Jacobian criterion in
		Remark~\ref{rem:jacobian-criterion}, $w$ is not a singular point of
		$\Spec R$.

		We have proved that the origin is the only possible singular point
		of $\Spec R$. Since $n-m\geq2$,
		Proposition~\ref{prop:normality-criterion} shows that $R$ is a normal
		domain. In particular, $I$ is prime.
	\end{proof}

	The bound $m\leq n-2$ in Theorem~\ref{thm:power-sum-progressions} cannot be improved.
	For $m=n-1$ the ring $S/I$ has dimension one, and $V(I)$ is a union of finitely many lines
	through the origin. If there is more than one line, then $I$ is not a prime ideal.
	For example, let $n=3$ and $I=(p_1,p_2)$. Then $V(I)$ is the union of the two lines through the points
	$(1,\omega,\omega^2)$ and $(1,\omega^2,\omega)$, where $\omega$ is a primitive cube root of unity.

	\subsection*{Adding one more power sum}

	We now determine when one more power sum $p_c$ can be added to the power sums $p_d,p_{2d},\ldots,p_{md}$.
	By Theorem~\ref{thm:power-sum-progressions}, the quotient by these power sums is a domain.
	So $p_c$ can be added exactly when $p_c$ does not lie in the ideal $(p_d,p_{2d},\ldots,p_{md})$.
	If $d$ divides $c$, we compute $p_c$ modulo this ideal; the same computation of coefficients
	appears in the proof of \cite[Theorem~2.11]{CKW}.
	If $d$ does not divide $c$, we use the following lemma.
	By a polynomial in $x_1^d,\ldots,x_n^d$ we mean an element of the subring
	$\C[x_1^d,\ldots,x_n^d]$ of $S$. This subring is the image of $T$ in Lemma~\ref{lem:finite-free-substitution}.

	\begin{lemma}\label{lem:nonmultiple-degree}
		Let $d,c\geq1$ be integers such that $d$ does not divide $c$.
		Let $I\subseteq S$ be an ideal that is generated by polynomials in $x_1^d,\ldots,x_n^d$.
		\begin{enumerate}[\rm(a)]
			\item If $p_c\in I$, then $x_i^c\in I$ for every $i$.
			\item If $I$ is a prime ideal and $I\ne\m$, then $p_c\notin I$.
		\end{enumerate}
	\end{lemma}

	\begin{proof}
		(a) Let $\zeta$ be a primitive $d$th root of unity, and fix an index $i$.
		Let $\sigma_i\colon S\longrightarrow S$ be the homomorphism of $\C$-algebras that sends
		$x_i$ to $\zeta x_i$ and does not change the other variables.
		Since $\sigma_i(x_i^d)=\zeta^dx_i^d=x_i^d$, we have $\sigma_i(g)=g$ for every polynomial $g$
		in $x_1^d,\ldots,x_n^d$.
		So $\sigma_i$ does not change the generators of $I$, and therefore $\sigma_i(I)\subseteq I$.
		Suppose that $p_c\in I$. Then $\sigma_i(p_c)\in I$, and so
		\[
		\sigma_i(p_c)-p_c=(\zeta^c-1)x_i^c\in I.
		\]
		Since $\zeta$ is a primitive $d$th root of unity and $d\nmid c$, we have $\zeta^c\ne1$.
		Hence $x_i^c\in I$.

		(b) Suppose that $I$ is a prime ideal and that $p_c\in I$.
		By (a), $x_i^c\in I$ for every $i$. Since $I$ is prime, $x_i\in I$ for every $i$.
		So $\m\subseteq I$. Since $\m$ is a maximal ideal and $I\ne S$, we get $I=\m$.
	\end{proof}

	We use part (b) in the following way; see the proofs of Theorem~\ref{thm:initial-semigroup} and Corollary~\ref{rem:quartic-triple-context}.
	Let $f_1,\ldots,f_r\in\C[x_1^d,\ldots,x_n^d]$ be homogeneous elements of positive degrees that
	form a regular sequence in $S$. Suppose that $r<n$ and that the ideal $I=(f_1,\ldots,f_r)$ is prime.
	Then $I\ne\m$, because $\hgt I=r<n$.
	So, if $d\nmid c$, then $p_c\notin I$ by Lemma~\ref{lem:nonmultiple-degree}(b), and
	$f_1,\ldots,f_r,p_c$ is a regular sequence by Remark~\ref{rem:regular-sequences}(d).

	\begin{Theorem}\label{thm:initial-semigroup}
		Let $S=\C[x_1,\ldots,x_n]$, let $1\leq m\leq n-2$, and let
		$d\geq1$ and $c\geq1$ be integers
		such that $c$ is not one of the degrees $d,2d,\ldots,md$. Put
		\[
		\Gamma_{m,n}=
		\left\{\sum_{j=m+1}^n j u_j\ \middle|\ u_j\in\N_0\right\}.
		\]
		Then $p_d,p_{2d},\ldots,p_{md},p_c$ form a regular sequence if and
		only if either $d\nmid c$, or $d\mid c$ and $c/d\in\Gamma_{m,n}$.
	\end{Theorem}

	\begin{proof}
		Put $I_d=(p_d,p_{2d},\ldots,p_{md})$.
		By Theorem~\ref{thm:power-sum-progressions} with $s=1$,
		the power sums $p_d,p_{2d},\ldots,p_{md}$ form a regular sequence, and $I_d$ is a prime ideal of height $m$.
		By Remark~\ref{rem:regular-sequences}(d), the sequence $p_d,p_{2d},\ldots,p_{md},p_c$ is regular
		if and only if $p_c\notin I_d$. So we must decide whether $p_c$ belongs to $I_d$.

		Suppose first that $d\nmid c$.
		The generators $p_d,p_{2d},\ldots,p_{md}$ of $I_d$ are polynomials in $x_1^d,\ldots,x_n^d$.
		Also $I_d\ne\m$, because $\hgt I_d=m<n$.
		So $p_c\notin I_d$ by Lemma~\ref{lem:nonmultiple-degree}(b), and the sequence is regular.

		Now suppose that $c=dq$ with an integer $q$.
		Then $q>m$, because $c$ is not one of the degrees $d,2d,\ldots,md$.
		The sequence $p_d,p_{2d},\ldots,p_{md},p_{dq}$ is obtained from $p_1,p_2,\ldots,p_m,p_q$ by the
		substitution $x_i\mapsto x_i^d$. By Lemma~\ref{lem:finite-free-substitution}(b), it is regular
		if and only if $p_1,p_2,\ldots,p_m,p_q$ is regular.
		By the first paragraph of the proof, applied with $d=1$ and $c=q$, this holds if and only if
		$p_q\notin I_1$, where $I_1=(p_1,p_2,\ldots,p_m)$. By Remark~\ref{rem:newton-identities}(b),
		\[
		I_1=(e_1,\ldots,e_m).
		\]
		We will compute the image of $p_q$ in $S/I_1$.

		The images of $e_{m+1},\ldots,e_n$ in $S/I_1$ are algebraically
		independent over $\C$.
		To see this, let $F$ be a polynomial in $n-m$ variables with complex coefficients such that
		$F(e_{m+1},\ldots,e_n)\in I_1$. Choose arbitrary complex numbers
		$\lambda_{m+1},\ldots,\lambda_n$, and let $w_1,\ldots,w_n$ be the roots, with multiplicities, of the polynomial
		\[
		X^n+\sum_{j=m+1}^n(-1)^j \lambda_j X^{n-j}.
		\]
		At the point $w=(w_1,\ldots,w_n)$ we have $e_j(w)=0$ for $1\leq j\leq m$ and
		$e_j(w)=\lambda_j$ for $m+1\leq j\leq n$.
		So every element of $I_1$ is zero at $w$, and therefore $F(\lambda_{m+1},\ldots,\lambda_n)=0$.
		This holds for all choices of the numbers $\lambda_j$. Hence $F$ is the zero polynomial.

		We use generating functions; see \cite[Chapter~I, \S2]{I-G}. Put
		\[
		D(t)=\prod_{i=1}^n(1-x_it)=1+\sum_{j=1}^n(-1)^je_jt^j.
		\]
		In the formal power series ring $S[[t]]$ we have $-\log(1-x_it)=\sum_{k\geq1}x_i^kt^k/k$ for every $i$.
		We take the sum over $i$ and get
		\begin{equation}\label{eq:log-generating}
			\sum_{k\geq1}\frac{p_k}{k}t^k=-\log D(t)
			=-\log\left(1+\sum_{j=1}^n(-1)^j e_jt^j\right).
		\end{equation}
		Now we work modulo $I_1$: we replace every coefficient by its image in $S/I_1$.
		The images of $e_1,\ldots,e_m$ are zero. So the right side becomes $-\log(1+\varepsilon(t))$, where
		$\varepsilon(t)=\sum_{j=m+1}^n(-1)^j e_jt^j$, and
		\[
		-\log\bigl(1+\varepsilon(t)\bigr)=\sum_{U\geq1}\frac{(-1)^U}{U}\,\varepsilon(t)^U.
		\]
		In the expansion of $\varepsilon(t)^U$, a choice of $u_j$ factors
		$(-1)^j e_jt^j$ for each $j$ contributes the multinomial coefficient
		$U!/\prod_j u_j!$. Thus, for a tuple $(u_{m+1},\ldots,u_n)$ of nonnegative integers with
		\[
		U=\sum_{j=m+1}^n u_j\geq1,
		\qquad \sum_{j=m+1}^n j u_j=q,
		\]
		the coefficient of $\prod_{j=m+1}^n e_j^{u_j}$ in the image of $p_q$ is
		\begin{equation}\label{eq:semigroup-coefficient}
			q(-1)^{q+U}\frac{(U-1)!}{\prod_{j=m+1}^n u_j!}.
		\end{equation}
		The sign comes from $(-1)^U$ in the logarithmic expansion and
		$(-1)^q$ in the product of the chosen terms.
		So the image of $p_q$ in $S/I_1$ is the sum of the monomials $\prod_{j=m+1}^ne_j^{u_j}$
		with the coefficients \eqref{eq:semigroup-coefficient}, over all such tuples.
		This is the classical formula of Girard and Waring for $p_q$ in terms of the elementary symmetric
		polynomials, with $e_1=\cdots=e_m=0$.

		The coefficients \eqref{eq:semigroup-coefficient} are nonzero rational numbers.
		Distinct tuples give distinct monomials in $e_{m+1},\ldots,e_n$, and the images of these elements are
		algebraically independent. So the terms cannot cancel.
		If there is at least one such tuple, then the image of $p_q$ is not zero, and $p_q\notin I_1$.
		If there is no such tuple, then the image of $p_q$ is zero, and $p_q\in I_1$.
		By the definition of $\Gamma_{m,n}$, a tuple $(u_{m+1},\ldots,u_n)$ of nonnegative integers with
		$\sum_{j=m+1}^nju_j=q$ exists if and only if $q\in\Gamma_{m,n}$.
		Hence $p_q\notin I_1$ if and only if $q\in\Gamma_{m,n}$.
		This proves the theorem in the case $d\mid c$.
	\end{proof}

	Here is an example for the last step of the proof. Let $(m,n)=(2,4)$, so that $I_1=(e_1,e_2)$.
	For $q=7$ the only tuple is $(u_3,u_4)=(1,1)$, and \eqref{eq:semigroup-coefficient} gives
	$p_7\equiv-7e_3e_4$ modulo $I_1$. So $p_1,p_2,p_7$ is a regular sequence.
	For $q=5$ there is no tuple, because $5$ is not a sum of copies of $3$ and $4$.
	So $p_5\in I_1$, and $p_1,p_2,p_5$ is not a regular sequence.

	\begin{citedremark}[{\cite[Lemma~1 and Corollary~5]{GSR}; \cite{Roberts}}]
		\label{rem:semigroup-interval}
		A set of nonnegative integers is called a numerical semigroup if it contains $0$, is closed under
		addition, and contains all large integers. The largest integer that does not belong to it is called
		its Frobenius number.
		The set $\Gamma_{m,n}$ in Theorem~\ref{thm:initial-semigroup} is the numerical semigroup generated by the
		integers $m+1,m+2,\ldots,n$. For such semigroups the following facts are known.
		\begin{enumerate}[\rm(a)]
			\item A sum of $k$ integers from $m+1,\ldots,n$ lies between $k(m+1)$ and $kn$, and every integer in this
			interval is such a sum; compare \cite[Lemma~1]{GSR}.
			So an integer $q\geq0$ belongs to $\Gamma_{m,n}$ if and only if $k(m+1)\leq q\leq kn$ for some integer $k\geq0$,
			that is, if and only if
			\[
			\left\lceil\frac qn\right\rceil\leq\left\lfloor\frac q{m+1}\right\rfloor.
			\]
			\item Put
			\[
			H=\left\lceil\frac{m}{n-m-1}\right\rceil.
			\]
			The Frobenius number of $\Gamma_{m,n}$ is $(m+1)H-1$.
			For $n\leq2m+1$ this is \cite[Corollary~5]{GSR}; for numerical semigroups that are generated by an arithmetic
			sequence, the Frobenius number was found by Roberts \cite{Roberts}.
			If $n\geq2m+1$, then $H=1$, and every integer $q\geq m+1$ belongs to $\Gamma_{m,n}$.
		\end{enumerate}
		So, if $c=dq$, the sequence $p_d,p_{2d},\ldots,p_{md},p_c$ of Theorem~\ref{thm:initial-semigroup} is regular
		if and only if $\lceil q/n\rceil\leq\lfloor q/(m+1)\rfloor$.
		For example, $(m,n)=(2,4)$ gives $H=2$ and the Frobenius number $5$,
		and $(m,n)=(3,5)$ gives $H=3$ and the Frobenius number $11$.
	\end{citedremark}

	\begin{Cor}\label{thm:triple-d-2d}
		Let $d\geq1$ and $c\geq1$ be integers with $c\ne d$ and $c\ne2d$. In $\C[x_1,x_2,x_3,x_4]$, the
		power sums $p_d,p_{2d},p_c$ form a regular sequence if and only
		if $c\ne5d$.
	\end{Cor}

	\begin{proof}
		Apply Theorem~\ref{thm:initial-semigroup} with $m=2$ and $n=4$.
		If $d\nmid c$, the sequence is regular. If $c=dq$, then $q\geq3$, and the sequence is regular if and only if
		$q\in\Gamma_{2,4}$. The semigroup $\Gamma_{2,4}$ consists of the sums of copies of $3$ and $4$.
		Among the integers $q\geq3$, only $5$ is not such a sum. Indeed, $3,4,6,7,8$ belong to $\Gamma_{2,4}$,
		and every integer at least $6$ is obtained from one of $6,7,8$
		by adding a nonnegative multiple of $3$.
		So, in the case $c=dq$, the sequence is regular if and only if $q\ne5$.
	\end{proof}

	The exceptional family $(d,2d,5d)$ is already recorded in
	\cite[Remark~2.16(c)]{CKW}, and \cite[Remark~2.2]{CSS} notes that $p_5\in(p_1,p_2)$ when $n\leq4$.
	The corollary shows that it is the only exception.

	For $d=1$,
	the exception has a simple reason, and this reason is not special to power sums.
	Let $f_1,f_2,f_5$ be homogeneous symmetric polynomials of degrees $1,2,5$ in four variables.
	Each of them is a polynomial in $e_1,e_2,e_3,e_4$.
	A monomial in $e_3$ and $e_4$ alone has degree $3u_3+4u_4$, and this is never $1$, $2$ or $5$.
	So every monomial of degree $1$, $2$ or $5$ in $e_1,e_2,e_3,e_4$ contains $e_1$ or $e_2$,
	and $f_1,f_2,f_5$ belong to the ideal $(e_1,e_2)$.
	This ideal has height two. Hence $\hgt(f_1,f_2,f_5)\leq2$, and $f_1,f_2,f_5$ is not a regular
	sequence, by the height criterion \eqref{eq:height-criterion}.
	This is the observation used in the proof of
	\cite[Theorem~2.5(iii)]{GGW}.
	Corollary~\ref{thm:triple-d-2d} shows that, for the power sums $p_d,p_{2d},p_c$, there is no other exception.

	\section{Two families in four variables}

	In this section, $S=\C[x_1,x_2,x_3,x_4]$. We consider the
	families $(d,3d,c)$ and $(d,5d,c)$.
	For the first family, we find the minimal primes of the ideal $(p_d,p_{3d})$.
	There are $3d^2$ of them, and each one is generated by two linear forms.
	With them we determine all degrees $c$ for which $p_d,p_{3d},p_c$ is a regular sequence.
	For the second family, Newton's identities show that the zero set of $(p_d,p_{5d})$ is the union of the
	zero sets of $(p_d,p_{2d})$ and $(p_d,p_{3d})$.
	So we get the minimal primes of $(p_d,p_{5d})$, and the criterion for $p_d,p_{5d},p_c$, from the
	first family and from Corollary~\ref{thm:triple-d-2d}.
	In both criteria, $c$ need not be a multiple of $d$.

	\begin{Prop}
		\label{thm:triple-d-3d}
		Let $d\geq1$ be an integer.
		\begin{enumerate}[\rm(a)]
			\item The ideal $(p_d,p_{3d})$ is the intersection of the $3d^2$ prime ideals
			\[
			(x_j-\zeta x_i,\ x_\ell-\zeta'x_k),
			\]
			where $(i,j,k,\ell)$ is one of $(1,2,3,4)$, $(1,3,2,4)$, $(1,4,2,3)$, and $\zeta$, $\zeta'$ are complex
			numbers with $\zeta^d=(\zeta')^d=-1$.
			These prime ideals are the minimal primes of $(p_d,p_{3d})$.
			So $V(p_d,p_{3d})$ is the union of $3d^2$ planes.
			\item Let $c\geq1$ be an integer.
			The power sum $p_c$ lies in the prime ideal $(x_j-\zeta x_i,\ x_\ell-\zeta'x_k)$ of {\rm(a)}
			if and only if $\zeta^c=(\zeta')^c=-1$.
			A prime ideal of {\rm(a)} that contains $p_c$ exists if and only if $\nu_2(d)=\nu_2(c)$.
			\item Let $c\geq1$ be an integer with $c\ne d$ and $c\ne3d$.
			Then $p_d,p_{3d},p_c$ form a regular sequence in $S$ if and only if $\nu_2(d)\ne\nu_2(c)$.
		\end{enumerate}
	\end{Prop}

	\begin{proof}
		(a) We first find the zero set of $(p_1,p_3)$.
		By Newton's identities \eqref{eq:newton-identities} for $k=3$, we have $3e_3=e_2p_1-e_1p_2+p_3$.
		Since $p_1=e_1$, this gives
		\[
		p_3\equiv3e_3\pmod{p_1}.
		\]
		Next, the following identity holds in $S$:
		\[
		e_3=e_1\,(x_1x_2+x_1x_3+x_2x_3)-(x_1+x_2)(x_1+x_3)(x_2+x_3).
		\]
		To check it, write $e_1=(x_1+x_2+x_3)+x_4$ and $e_3=x_1x_2x_3+x_4(x_1x_2+x_1x_3+x_2x_3)$, and use
		\[
		(x_1+x_2)(x_1+x_3)(x_2+x_3)=(x_1+x_2+x_3)(x_1x_2+x_1x_3+x_2x_3)-x_1x_2x_3.
		\]

		Let $z=(z_1,z_2,z_3,z_4)$ be a point of $\C^4$. We claim that $p_1(z)=p_3(z)=0$ if and only if
		\[
		z_i+z_j=0\quad\text{and}\quad z_k+z_\ell=0
		\]
		for one of the three choices of $(i,j,k,\ell)$ in the statement.
		Suppose first that $z_i+z_j=0$ and $z_k+z_\ell=0$.
		Then $p_1(z)=0$. Also $z_j^3=-z_i^3$ and $z_\ell^3=-z_k^3$, so $p_3(z)=0$.
		Conversely, suppose that $p_1(z)=p_3(z)=0$. The two formulas above give
		\[
		0=p_3(z)=-3\,(z_1+z_2)(z_1+z_3)(z_2+z_3).
		\]
		So one of the numbers $z_1+z_2$, $z_1+z_3$, $z_2+z_3$ is zero.
		Since $z_1+z_2+z_3+z_4=0$, the sum of the other two coordinates is then also zero.
		For example, if $z_2+z_3=0$, then $z_1+z_4=0$; this is the choice $(i,j,k,\ell)=(1,4,2,3)$.
		This proves the claim.

		Now let $w=(w_1,w_2,w_3,w_4)$ be a point of $\C^4$, and put $z=(w_1^d,w_2^d,w_3^d,w_4^d)$.
		Then $p_d(w)=p_1(z)$ and $p_{3d}(w)=p_3(z)$. By the claim, $w\in V(p_d,p_{3d})$ if and only if
		\[
		w_i^d+w_j^d=0\quad\text{and}\quad w_k^d+w_\ell^d=0
		\]
		for one of the three choices of $(i,j,k,\ell)$.
		There are exactly $d$ complex numbers $\zeta$ with $\zeta^d=-1$, and $X^d+1=\prod_{\zeta^d=-1}(X-\zeta)$.
		If we replace $X$ by $x_j/x_i$ and multiply by $x_i^d$, we get
		\[
		x_i^d+x_j^d=\prod_{\zeta^d=-1}(x_j-\zeta x_i).
		\]
		So $w_i^d+w_j^d=0$ if and only if $w_j=\zeta w_i$ for some $\zeta$ with $\zeta^d=-1$.
		We conclude that $V(p_d,p_{3d})$ is the union of the planes
		\[
		\{w\in\C^4:\ w_j=\zeta w_i,\ w_\ell=\zeta'w_k\},
		\]
		where $(i,j,k,\ell)$, $\zeta$ and $\zeta'$ are as in the statement. There are $3d^2$ choices.

		These $3d^2$ planes are pairwise different.
		To see this, take the plane given by $(i,j,k,\ell)$, $\zeta$ and $\zeta'$.
		On this plane, the coordinates $w_i$ and $w_j$ are either both zero or both nonzero,
		and the same holds for $w_k$ and $w_\ell$.
		The plane contains a point whose nonzero coordinates are exactly $w_i$ and $w_j$.
		On a plane that belongs to another choice of $(i,j,k,\ell)$ there is no such point.
		So the plane determines $(i,j,k,\ell)$, and then also the numbers $\zeta=w_j/w_i$ and $\zeta'=w_\ell/w_k$.

		We now pass to the ideals. Let $\mathfrak p=(x_j-\zeta x_i,\ x_\ell-\zeta'x_k)$ be one of the ideals in the statement.
		The ring $S/\mathfrak p$ is a polynomial ring in the images of $x_i$ and $x_k$.
		So $\mathfrak p$ is a prime ideal of height two, and $V(\mathfrak p)$ is the plane
		$\{w_j=\zeta w_i,\ w_\ell=\zeta'w_k\}$.
		By the Nullstellensatz, $\mathfrak p$ consists of all polynomials that are zero on this plane.
		Since the planes are pairwise different, the $3d^2$ prime ideals are pairwise different.
		By Theorem~\ref{rem:css-general}(c), the ideal $(p_d,p_{3d})$ is radical.
		So, again by the Nullstellensatz, it consists of all polynomials that are zero on $V(p_d,p_{3d})$.
		A polynomial is zero on a union of planes if and only if it is zero on each of them.
		Hence $(p_d,p_{3d})$ is the intersection of the $3d^2$ prime ideals $\mathfrak p$.
		A prime ideal that contains this intersection contains one of the ideals $\mathfrak p$.
		The ideals $\mathfrak p$ all have height two and are pairwise different, so none of them contains another one.
		Therefore they are exactly the minimal primes of $(p_d,p_{3d})$.

		(b) Let $\mathfrak p=(x_j-\zeta x_i,\ x_\ell-\zeta'x_k)$ be one of the prime ideals in (a).
		In $S/\mathfrak p$ we have $x_j=\zeta x_i$ and $x_\ell=\zeta'x_k$. So
		\[
		p_c=x_i^c+x_j^c+x_k^c+x_\ell^c\equiv\bigl(1+\zeta^c\bigr)x_i^c+\bigl(1+(\zeta')^c\bigr)x_k^c
		\pmod{\mathfrak p}.
		\]
		Since $S/\mathfrak p$ is a polynomial ring in the images of $x_i$ and $x_k$, this shows that
		$p_c\in\mathfrak p$ if and only if $\zeta^c=-1$ and $(\zeta')^c=-1$.

		It follows that a prime ideal of (a) that contains $p_c$ exists if and only if there is a complex number
		$\zeta$ with $\zeta^d=-1$ and $\zeta^c=-1$; if $\zeta$ is such a number, we can take $\zeta'=\zeta$.
		We show that such a number $\zeta$ exists if and only if $\nu_2(d)=\nu_2(c)$.
		Put $g=\gcd(d,c)$, $d=gd'$ and $c=gc'$. Then $\gcd(d',c')=1$, so $d'$ and $c'$ are not both even.
		Since $\nu_2(d)=\nu_2(g)+\nu_2(d')$ and $\nu_2(c)=\nu_2(g)+\nu_2(c')$, we have
		$\nu_2(d)=\nu_2(c)$ if and only if $d'$ and $c'$ are both odd.

		Suppose first that $d'$ and $c'$ are odd. Choose $\zeta\in\C$ with $\zeta^g=-1$.
		Then $\zeta^d=(\zeta^g)^{d'}=(-1)^{d'}=-1$, and in the same way $\zeta^c=-1$.
		Conversely, suppose that $\zeta^d=-1$ and $\zeta^c=-1$.
		There are integers $u$ and $v$ with $ud+vc=g$.
		So $\zeta^g=(\zeta^d)^u(\zeta^c)^v=(-1)^{u+v}$ is equal to $1$ or to $-1$.
		Since $(\zeta^g)^{d'}=\zeta^d=-1$, we get $\zeta^g=-1$, and $d'$ is odd.
		In the same way, $c'$ is odd.

		(c) By Theorem~\ref{rem:css-general}(a), $p_d,p_{3d}$ is a regular sequence.
		By Remark~\ref{rem:regular-sequences}(d), $p_d,p_{3d},p_c$ is a regular sequence if and only if
		$p_c$ lies in no minimal prime of $(p_d,p_{3d})$.
		By (a) and (b), this holds if and only if $\nu_2(d)\ne\nu_2(c)$.
	\end{proof}

	For example, for $d=1$ the proposition says that
	\[
	(p_1,p_3)=(x_1+x_2,\ x_3+x_4)\cap(x_1+x_3,\ x_2+x_4)\cap(x_1+x_4,\ x_2+x_3),
	\]
	and that $p_1,p_3,p_c$ is a regular sequence if and only if $c$ is even.
	The inclusion $(p_1,p_3)\subseteq(x_1+x_2,\ x_3+x_4)$ is noted in \cite[Remark~3.3(1)]{CSS};
	Proposition~\ref{thm:triple-d-3d} gives all minimal primes.

	For $c>3d$, the criterion in (c) can also be deduced from known results in three variables.
	By Lemma~\ref{lem:finite-free-substitution}(b), applied to the substitution $x_i\mapsto x_i^g$ with
	$g=\gcd(d,c)$, we may assume that $\gcd(d,c)=1$. Then the condition $\nu_2(d)\ne\nu_2(c)$ means that $dc$ is even.
	If $dc$ is even, then $p_d,p_{3d},p_c$ form a regular sequence in three variables
	by \cite[Proposition~2.13]{CKW}. So the origin is the only common zero of $p_d$, $p_{3d}$, $p_c$ and $x_4$
	in $\C^4$, and the ideal $(p_d,p_{3d},p_c,x_4)$ of $S$ has height four.
	One generator raises the height by at most one. Hence $(p_d,p_{3d},p_c)$ has height three,
	and $p_d,p_{3d},p_c$ is a regular sequence in $S$.
	If $dc$ is odd, all three power sums vanish on the plane $x_2=-x_1$, $x_4=-x_3$, so their ideal has
	height at most two; compare \cite[Lemma~2.4(a)]{CKW}.

	\begin{Cor}\label{cor:triple-d-5d}
		Let $d\geq1$ be an integer.
		\begin{enumerate}[\rm(a)]
			\item We have
			\[
			(p_d,p_{5d})=(p_d,p_{2d})\cap(p_d,p_{3d}).
			\]
			The minimal primes of $(p_d,p_{5d})$ are the prime ideal $(p_d,p_{2d})$ and the $3d^2$ prime ideals
			of Proposition~\ref{thm:triple-d-3d}{\rm(a)}.
			\item Let $c\geq1$ be an integer with $c\ne d$ and $c\ne5d$.
			Then $p_d,p_{5d},p_c$ form a regular sequence in $S$ if and only if $c\ne2d$ and $\nu_2(d)\ne\nu_2(c)$.
		\end{enumerate}
	\end{Cor}
	\begin{proof}
		(a) We use Newton's identities \eqref{eq:newton-identities} for $k=2$, $3$ and $5$.
		Since $n=4$, we have $e_5=0$, and the identities give
		\[
		p_2=e_1p_1-2e_2,\qquad p_3=e_1p_2-e_2p_1+3e_3,\qquad
		p_5=e_1p_4-e_2p_3+e_3p_2-e_4p_1.
		\]
		Recall that $p_1=e_1$.
		In the quotient $S/(p_1)=S/(e_1)$, we therefore get
		\[
		p_2=-2e_2,\qquad p_3=3e_3,\qquad p_5=-5e_2e_3.
		\]
		For the last formula, note that $p_5=-e_2p_3+e_3p_2=-3e_2e_3-2e_2e_3$ in this quotient.
		In particular, $p_5\in(p_1,p_3)$, as recorded in
		\cite[Remark~2.2]{CSS}.
		Now let $z$ be a point of $\C^4$ with $p_1(z)=0$. Then $p_5(z)=-5e_2(z)e_3(z)$.
		So $p_5(z)=0$ if and only if $e_2(z)=0$ or $e_3(z)=0$.
		By the first two formulas, this means that $p_2(z)=0$ or $p_3(z)=0$. Consequently,
		\[
		V(p_1,p_5)=V(p_1,p_2)\cup V(p_1,p_3).
		\]
		Let $w=(w_1,w_2,w_3,w_4)$ be a point of $\C^4$, and put $z=(w_1^d,w_2^d,w_3^d,w_4^d)$.
		Then $p_{kd}(w)=p_k(z)$ for every $k\geq1$.
		So $w$ lies in $V(p_d,p_{kd})$ if and only if $z$ lies in $V(p_1,p_k)$.
		We apply the last displayed equality to the point $z$, and we get
		\begin{equation}\label{eq:triple-five-decomposition}
			V(p_d,p_{5d})=V(p_d,p_{2d})\cup V(p_d,p_{3d}).
		\end{equation}
		We now pass to the ideals.
		The ideals $(p_d,p_{5d})$ and $(p_d,p_{3d})$ are radical by Theorem~\ref{rem:css-general}(c),
		and $(p_d,p_{2d})$ is a prime ideal of height two by Theorem~\ref{thm:power-sum-progressions}.
		By the Nullstellensatz, a radical ideal of $S$ consists of all polynomials that are zero on its zero set.
		A polynomial is zero on a union of two sets if and only if it is zero on both sets.
		So \eqref{eq:triple-five-decomposition} gives
		\[
		(p_d,p_{5d})=(p_d,p_{2d})\cap(p_d,p_{3d}).
		\]
		By Proposition~\ref{thm:triple-d-3d}(a), the ideal $(p_d,p_{5d})$ is therefore the intersection of
		$(p_d,p_{2d})$ and the $3d^2$ prime ideals of that proposition. All these prime ideals have height two.
		The prime ideal $(p_d,p_{2d})$ is different from the other ones:
		it contains $p_{2d}$, and the other ones do not contain $p_{2d}$,
		by Proposition~\ref{thm:triple-d-3d}(b), because $\nu_2(2d)\ne\nu_2(d)$.
		So no prime ideal in this list contains another one, and the list consists of the minimal primes
		of $(p_d,p_{5d})$.

		(b) By Theorem~\ref{rem:css-general}(a), $p_d,p_{5d}$ is a regular sequence.
		By Remark~\ref{rem:regular-sequences}(d), $p_d,p_{5d},p_c$ is a regular sequence if and only if
		$p_c$ lies in no minimal prime of $(p_d,p_{5d})$. The minimal primes are listed in (a).
		By Proposition~\ref{thm:triple-d-3d}(b), $p_c$ lies in one of the $3d^2$ prime ideals of
		Proposition~\ref{thm:triple-d-3d}(a) if and only if $\nu_2(d)=\nu_2(c)$.

		It remains to decide when $p_c$ lies in $(p_d,p_{2d})$. This is the case if $c=2d$.
		Suppose that $c\ne2d$. Then $c$ is different from $d$, $2d$ and $5d$.
		So $p_d,p_{2d},p_c$ is a regular sequence by Corollary~\ref{thm:triple-d-2d}.
		Since $(p_d,p_{2d})$ is a prime ideal, Remark~\ref{rem:regular-sequences}(d) shows that
		$p_c\notin(p_d,p_{2d})$.
		We conclude that $p_c$ lies in no minimal prime of $(p_d,p_{5d})$ if and only if
		$c\ne2d$ and $\nu_2(d)\ne\nu_2(c)$.
	\end{proof}

	\begin{remark}[Comparison with the three-variable result]

		Let $d\geq1$ and $c>5d$ be integers.
		By \cite[Lemma~2.2 and Proposition~2.13]{CKW},
		the power sums $p_d,p_{5d},p_c$ form a regular sequence in
		$\C[x_1,x_2,x_3]$ if and only if
		\[
		\nu_2(d)\ne\nu_2(c)\quad\text{and}\quad\nu_3(d)\ne\nu_3(c).
		\]
		To see this, put $g=\gcd(d,c)$, $d=gd'$ and $c=gc'$.
		By \cite[Lemma~2.2]{CKW}, we may divide the three degrees by $g$.
		By \cite[Proposition~2.13]{CKW}, the power sums $p_{d'},p_{5d'},p_{c'}$ form a regular sequence
		in three variables if and only if $6$ divides the product $d'\cdot5d'\cdot c'$, that is,
		if and only if $6$ divides $d'c'$.
		The integers $d'$ and $c'$ have no common prime factor.
		So $6$ divides $d'c'$ if and only if exactly one of $d',c'$ is even and exactly one of them is divisible by $3$.
		These are the two displayed conditions.
		Corollary~\ref{cor:triple-d-5d} shows that the second condition is
		not needed in four variables. For example, $p_2,p_{10},p_{13}$
		form a regular sequence in four variables but not in three.
		Indeed, in three variables they vanish at $(1,\omega,\omega^2)$,
		where $\omega$ is a primitive cube root of unity.
		For the family $(d,3d,c)$ there is no such difference.
		By the same two results of \cite{CKW}, for $c>3d$ the power sums $p_d,p_{3d},p_c$ form a regular sequence in
		$\C[x_1,x_2,x_3]$ if and only if $\nu_2(d)\ne\nu_2(c)$.
		This is the condition of Proposition~\ref{thm:triple-d-3d}(c).
	\end{remark}
	\section{The ideals \texorpdfstring{$(p_d,p_{sd})$}{(p-d, p-sd)} with \texorpdfstring{$s$}{s} even}
	\label{sec:quartic-pair}

	Throughout this section, $S=\C[x_1,x_2,x_3,x_4]$.
	Let $d\geq1$ be an integer, and let $s\geq2$ be an even integer.
	We prove that the ideal $(p_d,p_{sd})$ is prime, with one exception:
	if $s=4$ and $d$ is even, then $(p_d,p_{4d})$ has exactly two minimal primes, and we determine them.
	We begin with a theorem of Conca, Singh and Soundararajan on two power sums in four variables.
	It shows which of these ideals were already known to be prime.
	Our proof uses the plane curve $p_1=p_s=0$, which is the case $d=1$, and the zeros of the
	functions $y_i/y_4$ on this curve.
	The case $s=4$ is special, because the curve $p_1=p_4=0$ is the only one of these curves on which
	every zero of the functions $y_i/y_4$ is a double zero.
	In the second part of the section we use the minimal primes of $(p_d,p_{4d})$.
	We show that $p_d,p_{4d},p_c$ is a regular sequence for every degree $c$ different from $d$ and $4d$.
	At the end we collect our conditions for three power sums in a table, and we state a question.

	\subsection*{Two power sums in four variables}

	The following theorem decides when the quotient by two power sums in four variables is a normal domain.
	We use the notation $\nu_\ell$ of Remark~\ref{rem:notation}(d).

	\begin{citedtheorem}[{\cite[Theorem~3.1]{CSS}}]
		\label{lem:pair-normality}
		Suppose that $n=4$, and let $a<b$ be positive integers.
		\begin{enumerate}[\rm(a)]
			\item If $a=1$, then
			$S/(p_a,p_b)$ is a normal domain if and only if $b$ is even.
			\item If $a>1$, then $S/(p_a,p_b)$ is a
			normal domain if and only if both of the following conditions hold:
			\[
			\begin{gathered}
				\nu_2(a)\ne\nu_2(b),\\
				\nu_3(a)\ne\nu_3(b)
				\quad\text{or}\quad
				\nu_3(a)=\nu_3(b)=\nu_3(b-a).
			\end{gathered}
			\]
		\end{enumerate}
	\end{citedtheorem}

	We use part (b) in the following equivalent form, which is easier to apply:
	\begin{equation}\label{eq:pair-normality-gcd}
		S/(p_a,p_b)\text{ is a normal domain}
		\quad\Longleftrightarrow\quad
		\gcd\left(\frac{b-a}{\gcd(a,b)},6\right)=1
		\qquad(a>1).
	\end{equation}

	The two forms are equivalent.
	To see this, let $1<a<b$, and put $g=\gcd(a,b)$, $a=ga'$ and $b=gb'$.
	Then $\gcd(a',b')=1$, and $(b-a)/g=b'-a'$; we write $L=b'-a'$.
	The integers $a'$ and $b'$ are not both even.
	So $L$ is odd exactly when one of $a',b'$ is even.
	Since $\nu_2(a)=\nu_2(g)+\nu_2(a')$ and $\nu_2(b)=\nu_2(g)+\nu_2(b')$, this means that
	$\nu_2(a)\ne\nu_2(b)$.
	In the same way, $a'$ and $b'$ are not both divisible by $3$.
	If $\nu_3(a)\ne\nu_3(b)$, then exactly one of $a',b'$ is divisible by $3$, and so $3\nmid L$.
	If $\nu_3(a)=\nu_3(b)$, call this number $v$.
	Then $3$ divides neither $a'$ nor $b'$, so $\nu_3(g)=v$ and $\nu_3(b-a)=v+\nu_3(L)$.
	In this case $3\nmid L$ exactly when $\nu_3(b-a)=v$.
	Hence $\gcd(L,6)=1$ if and only if both conditions of Theorem~\ref{lem:pair-normality}(b) hold.

	The proof of Theorem~\ref{lem:pair-normality} in \cite{CSS} uses the Jacobian criterion.
	The singular points of $\Spec\bigl(S/(p_a,p_b)\bigr)$ other than the origin come from sums of at most four
	roots of unity that are equal to zero; these sums are treated in \cite[Lemma~3.2]{CSS}.
	For vanishing sums of roots of unity in general, see \cite[Theorem~5.2]{LL} and \cite{KSK}.

	\subsection*{The minimal primes of \texorpdfstring{$(p_d,p_{sd})$}{(p-d, p-sd)}}

	We first recall two facts. The first fact is about smooth plane curves of degree four.
	The second fact is about the algebra that we get from a field $K$ when we adjoin $d$th roots of
	some elements of $K$.
	Then we state and prove the theorem on the minimal primes of $(p_d,p_{sd})$.

	\begin{citedremark}[{\cite[Chapter~IV, Proposition~5.2 and Example~5.2.1]{Hartshorne}}]
		\label{rem:plane-quartic}
		A smooth plane quartic, that is, a smooth curve of degree four in the projective plane, has genus three and is not hyperelliptic.
		In particular, there is no morphism of degree two from such a curve to $\mathbb P^1$.
		We will use the following consequence. A nonconstant rational function
		on such a curve cannot have a pole divisor of degree two, since that
		function would define a morphism of degree two to $\mathbb P^1$.
		Here the pole divisor of a rational function is the sum of its poles, and each pole is counted with its order.
	\end{citedremark}

	\begin{citedprop}[{(Kummer theory) \cite[Chapter~VI, \S8]{Lang}}]
		\label{rem:root-components}
		Let $d,r\geq1$ be integers, and let $K$ be a field of characteristic zero that
		contains all $d$th roots of unity. Let $f_1,\ldots,f_r\in K^*$, and put
		\[
		B=K[z_1,\ldots,z_r]/(z_1^d-f_1,\ldots,z_r^d-f_r).
		\]
		Then $B$ is a finite product of fields. The number of its field factors is
		\begin{equation}\label{eq:root-component-count}
			\#\left\{(\alpha_1,\ldots,\alpha_r)\in(\mathbb Z/d\mathbb Z)^r\ \middle|\
			\prod_{i=1}^r f_i^{\alpha_i}\in K^{*d}\right\}.
		\end{equation}
		Here $K^{*d}=\{\psi^d:\psi\in K^*\}$. The condition is independent
		of the chosen representatives of the $\alpha_i$.
		In particular, $B$ is a field if and only if $(0,\ldots,0)$ is the only element of the set
		in \eqref{eq:root-component-count}.
	\end{citedprop}

	\begin{proof}
		Let $\overline K$ be an algebraic closure of $K$. For each $i$, choose an element $\beta_i\in\overline K$
		with $\beta_i^d=f_i$, and put $K'=K(\beta_1,\ldots,\beta_r)$.
		There are exactly $d$ different $d$th roots of unity, because the characteristic is zero,
		and they lie in $K$.
		So the polynomial $z_i^d-f_i$ has the $d$ different roots $\zeta\beta_i$ with $\zeta^d=1$, and all of them lie
		in $K'$.

		\emph{The ring $B$ is a finite product of fields.}
		By the Chinese remainder theorem, $K'[z_i]/(z_i^d-f_i)$ is a product of $d$ copies of $K'$.
		We use this for $z_1,\ldots,z_r$, one variable after the other. We get that the ring
		\[
		B\otimes_KK'=K'[z_1,\ldots,z_r]/(z_1^d-f_1,\ldots,z_r^d-f_r)
		\]
		is a product of $d^r$ copies of $K'$. So it has no nilpotent element other than zero, and the same holds for its
		subring $B$. As a vector space over $K$, the ring $B$ has the basis $z_1^{u_1}\cdots z_r^{u_r}$ with
		$0\leq u_i<d$; so $\dim_KB=d^r$.
		A ring of finite dimension over a field that has no nilpotent element other than zero is a finite product of
		fields.

		\emph{Every field factor of $B$ is isomorphic to $K'$.}
		Let $B'$ be a field factor of $B$. It is a finite extension of $K$, generated by the images of
		$z_1,\ldots,z_r$. So there is an injective homomorphism of $K$-algebras from $B'$ to $\overline K$.
		It sends the image of $z_i$ to $\zeta_i\beta_i$ for some $\zeta_i$ with $\zeta_i^d=1$.
		Since $\zeta_i\in K$, the image of $B'$ is $K(\zeta_1\beta_1,\ldots,\zeta_r\beta_r)=K'$.
		We compare the dimensions over $K$ and get
		\[
		(\text{number of field factors of }B)\cdot[K':K]=d^r.
		\]

		\emph{The degree of $K'$ over $K$.}
		Let $W$ be the subgroup of the group $K^*/K^{*d}$ that is generated by the classes of $f_1,\ldots,f_r$.
		By Kummer theory, $[K':K]=|W|$, the number of elements of $W$ \cite[Chapter~VI, \S8]{Lang}.
		The map
		\[
		(\mathbb Z/d\mathbb Z)^r\longrightarrow W,\qquad
		(\alpha_1,\ldots,\alpha_r)\longmapsto\text{the class of }f_1^{\alpha_1}\cdots f_r^{\alpha_r},
		\]
		is well defined, because the $d$th power of every element of $K^*/K^{*d}$ is the neutral element.
		It is a surjective homomorphism of groups, and the set in \eqref{eq:root-component-count} is its kernel.
		So this set has $d^r/|W|$ elements, and this is the number of field factors of $B$.
	\end{proof}

	Instead of Kummer theory, one can also use the linear independence theorem of Carr and O'Sullivan
	\cite[Theorem~1.3]{CO}: it shows that $[K':K]\geq|W|$, and the inequality $[K':K]\leq|W|$ is elementary.

	In the proof of Theorem~\ref{thm:quartic-pair-components},
	we use Proposition~\ref{rem:root-components} with $r=3$.
	There $K$ is the field of rational functions on the curve $p_1=p_s=0$, and the number of field factors of $B$
	is the number of minimal primes of $(p_d,p_{sd})$.

	\begin{Theorem}\label{thm:quartic-pair-components}
		Let $S=\C[x_1,x_2,x_3,x_4]$, let $d\geq1$ be an integer, let $s\geq2$ be an even integer, and put $J=(p_d,p_{sd})$.
		\begin{enumerate}[\rm(a)]
			\item If $s\ne4$, or if $d$ is odd, then $J$ is a prime ideal.
			\item If $s=4$ and $d$ is even, put
			\[
			E_2=e_2(x_1^d,x_2^d,x_3^d,x_4^d),\qquad
			G=(x_1x_2x_3x_4)^{d/2}.
			\]
			Then $J=(p_d,p_{4d})$ has exactly two minimal primes, namely
			\[
			Q_+=(p_d,E_2+\sqrt2\,G),\qquad
			Q_-=(p_d,E_2-\sqrt2\,G),
			\]
			and $J=Q_+\cap Q_-$.
		\end{enumerate}
	\end{Theorem}

	Some cases of Theorem~\ref{thm:quartic-pair-components} are known.
	Part~(a) is known for $d=1$: the ideal $(p_1,p_s)$ is prime by \cite[Proposition~4.3]{KM}, and $S/(p_1,p_s)$
	is a normal domain by Theorem~\ref{lem:pair-normality}(a).
	Part~(a) is also known if $s\not\equiv4\pmod6$, because then $S/(p_d,p_{sd})$ is a normal domain by
	Theorem~\ref{lem:pair-normality}(b); see Remark~\ref{rem:quartic-normality}(a).
	For $s=4$ and even $d$, the ideal $J$ is not prime by \cite[Remark~3.3(2) and~(6)]{CSS}, and for $d=2$ the element
	$E_2-\sqrt2\,G$ is, modulo $p_2$, the zero-divisor of $S/(p_2,p_8)$ in \cite[Remark~3.3(6)]{CSS}.
	The new results are part~(a) for $d\geq2$ and $s\equiv4\pmod6$, and, in part~(b), that $Q_+$ and $Q_-$ are
	prime and are the only minimal primes.

	\begin{proof}
		The proof has eight steps. Steps~1 to~6 hold for every even integer $s\geq2$.
		In Step~7 we prove {\rm(a)}, and in Step~8 we prove {\rm(b)}.

		\emph{Step 1. What we have to show.}
		By Theorem~\ref{rem:css-general}(a) and~(c),
		the power sums $p_d,p_{sd}$ form a regular sequence, and the ideal $J$ is radical.
		So $J$ is the intersection of its minimal primes, and each minimal prime of $J$ has height two
		(Remark~\ref{rem:regular-sequences}(b)).
		Hence $J$ is a prime ideal if and only if it has exactly one minimal prime.
		In this proof we count the minimal primes of $J$.

		It is convenient to use geometric language.
		The ideal $J$ is homogeneous, and so are its minimal primes.
		Let $C_d$ be the set of all points $[x_1:x_2:x_3:x_4]$ of the projective space $\mathbb P^3$ with
		$p_d(x)=p_{sd}(x)=0$.
		The irreducible components of $C_d$ are the zero sets in $\mathbb P^3$ of the minimal primes of $J$.
		Each of them is a curve, because the minimal primes have height two.
		So the number of minimal primes of $J$ is the number of irreducible components of the curve $C_d$.

		We will also use that no minimal prime of $J$ contains one of the variables $x_1,x_2,x_3,x_4$.
		To see this for $x_4$, note that $S/(x_4)$ is the polynomial ring $\C[x_1,x_2,x_3]$, and that the images of
		$p_d$ and $p_{sd}$ in this ring are the power sums of degrees $d$ and $sd$ in three variables.
		By Theorem~\ref{rem:css-general}(a), with $n=3$ and $r=2$, these two power sums generate an ideal of
		height two in $\C[x_1,x_2,x_3]$.
		So the ideal $(p_d,p_{sd},x_4)$ of $S$ has height three.
		A minimal prime of $J$ has height two, and therefore it does not contain $x_4$.
		The same argument works for $x_1$, $x_2$ and $x_3$.
		In geometric language: no irreducible component of $C_d$ lies in one of the four planes $x_i=0$.

		\emph{Step 2. The curve $C$.}
		We first study the case $d=1$.
		To keep this case apart from the general case, we write $y_1,\ldots,y_4$ for the coordinates.
		In projective coordinates $[y_1:y_2:y_3:y_4]$, put
		\[
		C=\{p_1(y)=p_s(y)=0\}\subseteq\mathbb P^3.
		\]
		The equation $p_1(y)=0$ defines a plane in $\mathbb P^3$, and $C$ is
		a curve of degree $s$ in this plane. For $s=4$ it is a plane quartic.
		Since $s$ is even, $(p_1,p_s)$ is a prime ideal by Theorem~\ref{lem:pair-normality}(a).
		Hence $C$ is irreducible.
		The curve $C$ is smooth.
		Indeed, a point of $C$ is a singular point exactly when the Jacobian matrix of $p_1,p_s$ has rank less than two
		at this point; compare Remark~\ref{rem:jacobian-criterion}.
		Let $P$ be a point of $C$, and let $\mathfrak p$ be the ideal of all polynomials in $y_1,\ldots,y_4$ that are zero
		on the line through the origin that represents $P$.
		Then $\mathfrak p$ is a prime ideal that contains $p_1$ and $p_s$, and the quotient by $\mathfrak p$ has
		dimension one.
		By the last formula in Remark~\ref{rem:regular-sequences}(b), the ideal $\mathfrak p$ gives a prime ideal of
		height one in the ring $R=\C[y_1,\ldots,y_4]/(p_1,p_s)$.
		By Theorem~\ref{lem:pair-normality}(a), the ring $R$ is a normal domain.
		So it satisfies $(R_1)$ (Serre's criterion; see Section~\ref{sec:power-sum-progressions}), and the local ring
		$R_{\mathfrak p}$ is regular.
		By the Jacobian criterion in Remark~\ref{rem:jacobian-criterion}, some $2\times2$ minor of the Jacobian matrix
		does not lie in $\mathfrak p$; that is, it is not zero at some point of the line.
		The minors are homogeneous. So this minor is not zero at $P$, and the Jacobian matrix has rank two at $P$.
		Hence $P$ is not a singular point of $C$.

		Let $K=\C(C)$ be the field of rational functions on $C$.
		Since $C$ is a smooth curve, the local ring $\mathcal O_P$ of $C$ at a point $P$ is a discrete valuation ring
		with field of fractions $K$.
		For $\varphi\in K^*$, we write $\operatorname{ord}_P(\varphi)$ for the value of $\varphi$ under this valuation;
		it is an integer.
		The function $\varphi$ has a zero of order $\operatorname{ord}_P(\varphi)$ at $P$ if this integer is positive,
		and it has a pole of order $-\operatorname{ord}_P(\varphi)$ at $P$ if this integer is negative.
		We have $\operatorname{ord}_P(\varphi\psi)=\operatorname{ord}_P(\varphi)+\operatorname{ord}_P(\psi)$
		for all $\varphi,\psi\in K^*$, and $\operatorname{ord}_P(\varphi)=0$ if $\varphi$ is defined at $P$ and
		$\varphi(P)\ne0$.
		We also need the following fact: if $\varphi\ne0$ lies in $\mathcal O_P$, then
		\[
		\operatorname{ord}_P(\varphi)=\dim_\C\mathcal O_P/(\varphi).
		\]
		To see this, let $\tau$ be a generator of the maximal ideal of $\mathcal O_P$, and put
		$\mu=\operatorname{ord}_P(\varphi)$. Then $\varphi$ is the product of $\tau^\mu$ and a unit, and
		$\mathcal O_P/(\tau^\mu)$ has the basis $1,\tau,\ldots,\tau^{\mu-1}$ as a vector space over $\C$.

		\emph{Step 3. The points of $C$ in the coordinate planes.}
		Put
		\[
		\gamma(X)=1+X^s+(1+X)^s.
		\]
		This is a polynomial of degree $s$ in one variable $X$.
		Since $s$ is even, we have $\gamma(0)=2$ and $\gamma(-1)=2$. So $0$ and $-1$ are not roots of $\gamma$.
		We also have $\gamma(-1-X)=\gamma(X)$.
		So, if $\lambda$ is a root of $\gamma$, then $-1-\lambda$ is a root of $\gamma$ with the same multiplicity.

		Let $P=[y_1:y_2:y_3:y_4]$ be a point of $C$ with $y_4=0$.
		Then $y_1+y_2+y_3=0$ and $y_1^s+y_2^s+y_3^s=0$.
		If $y_1=0$, then $y_3=-y_2$ and $2y_2^s=0$, so all coordinates are zero. This is impossible.
		So $y_1\ne0$, and we may assume that $y_1=1$. Put $\lambda=y_2$.
		Then $y_3=-1-\lambda$, and the second equation says that $\gamma(\lambda)=0$, because
		$(-1-\lambda)^s=(1+\lambda)^s$.
		Conversely, every root $\lambda$ of $\gamma$ gives a point of $C$ with $y_4=0$.
		So the points of $C$ in the plane $y_4=0$ are the points
		\[
		[1:\lambda:-1-\lambda:0]\qquad\text{with }\gamma(\lambda)=0.
		\]
		Since $\lambda\ne0$ and $\lambda\ne-1$, the coordinates $y_1,y_2,y_3$ are not zero at these points.

		Let $P=[1:\lambda:-1-\lambda:0]$ be such a point. We claim that
		\[
		\operatorname{ord}_P(y_4/y_1)=\text{the multiplicity of $\lambda$ as a root of $\gamma$.}
		\]
		To prove this, we look at the part of the plane $p_1(y)=0$ where $y_1\ne0$.
		There we may assume that $y_1=1$, and we use $y_2$ and $y_4$ as coordinates; then $y_3=-1-y_2-y_4$.
		In these coordinates, $C$ is the curve $\Psi(y_2,y_4)=0$, where
		\[
		\Psi(y_2,y_4)=1+y_2^s+(1+y_2+y_4)^s+y_4^s,
		\]
		and $P$ is the point $(\lambda,0)$.
		Because $(p_1,p_s)$ is a prime ideal, the ring of polynomial functions on this part of $C$ is
		$\C[y_2,y_4]/(\Psi)$, and $\mathcal O_P$ is the localization of this ring at the maximal ideal
		$(y_2-\lambda,\,y_4)$.
		We have $\Psi(y_2,0)=\gamma(y_2)$. So
		\[
		\C[y_2,y_4]/(\Psi,\,y_4)=\C[y_2]/(\gamma(y_2)).
		\]
		Let $\mu$ be the multiplicity of the root $\lambda$, and write $\gamma(y_2)=(y_2-\lambda)^\mu\gamma_1(y_2)$,
		where $\gamma_1$ is a polynomial with $\gamma_1(\lambda)\ne0$.
		When we localize at the maximal ideal $(y_2-\lambda)$, the polynomial $\gamma_1$ becomes a unit. Hence
		\[
		\mathcal O_P/(y_4)\cong\C[y_2]/\bigl((y_2-\lambda)^\mu\bigr),
		\]
		and this is a vector space of dimension $\mu$ over $\C$.
		By the fact at the end of Step~2, the function $y_4$ has order $\mu$ at $P$.
		In our coordinates, the function $y_4$ is $y_4/y_1$. This proves the claim.

		The polynomials $p_1$ and $p_s$ are symmetric.
		So all that we said about the coordinates $y_4$ and $y_1$ is also true for every other pair of coordinates.
		We state this for later use.
		Let $i\ne j$ be two of the indices $1,2,3,4$, and let $P$ be a point of $C$ with $y_i(P)=0$.
		Then the other three coordinates of $P$ are not zero. In particular, every point of $C$ has at most one
		coordinate equal to zero.
		If we divide the coordinates of $P$ by $y_j(P)$, then the $j$th coordinate is $1$, and the other two nonzero coordinates
		are $\lambda$ and $-1-\lambda$ for a root $\lambda$ of $\gamma$. Moreover,
		\[
		\operatorname{ord}_P(y_i/y_j)=\text{the multiplicity of $\lambda$ as a root of $\gamma$.}
		\]
		Conversely, every point of this form lies on $C$.

		\emph{Step 4. The roots of $\gamma$.}
		We prove three facts.

		\emph{First fact: every multiple root of $\gamma$ is a primitive cube root of unity, and its multiplicity is two.}
		The derivative of $\gamma$ is $\gamma'(X)=s\bigl(X^{s-1}+(1+X)^{s-1}\bigr)$.
		Let $\lambda$ be a multiple root of $\gamma$, so that $\gamma(\lambda)=0$ and $\gamma'(\lambda)=0$.
		The second equation gives $(1+\lambda)^{s-1}=-\lambda^{s-1}$. Then
		\[
		0=\gamma(\lambda)=1+\lambda\cdot\lambda^{s-1}+(1+\lambda)(1+\lambda)^{s-1}
		=1+\lambda\cdot\lambda^{s-1}-(1+\lambda)\lambda^{s-1}=1-\lambda^{s-1}.
		\]
		So $\lambda^{s-1}=1$ and $(1+\lambda)^{s-1}=-1$.
		It follows that the complex numbers $\lambda$ and $1+\lambda$ have absolute value $1$.
		Since $|1+\lambda|^2=1+2\operatorname{Re}(\lambda)+|\lambda|^2$, the real part of $\lambda$ is $-1/2$.
		So $\lambda=-\frac12\pm\frac{\sqrt3}2\,\sqrt{-1}$, and these two numbers are the primitive cube roots of unity.
		The second derivative of $\gamma$ is $\gamma''(X)=s(s-1)\bigl(X^{s-2}+(1+X)^{s-2}\bigr)$.
		With $\lambda^{s-1}=1$ and $(1+\lambda)^{s-1}=-1$ we get
		\[
		\gamma''(\lambda)=s(s-1)\left(\frac1\lambda-\frac1{1+\lambda}\right)=\frac{s(s-1)}{\lambda(1+\lambda)}\ne0.
		\]
		So the multiplicity of the root $\lambda$ is exactly two.
		The first fact is also in the literature. For even $s$, Nanninga writes
		$\gamma=(X^2+X+1)^bT_s$, with $b=0,1,2$ for $s\equiv0,2,4\pmod6$, and he shows that all roots of the
		polynomial $T_s$ are simple \cite[\S2 and Lemma~6]{Nanninga}; see also \cite[Lemma~4.3]{CSS}.
		We include the short proof above to keep the argument complete.

		\emph{Second fact: if $s\ne4$, then $\gamma$ has a simple root.}
		Suppose that $\gamma$ has no simple root.
		By the first fact, every root of $\gamma$ is then a primitive cube root of unity, with multiplicity two.
		There are only two primitive cube roots of unity. So the degree $s$ of $\gamma$ is at most four.
		For $s=2$ we have $\gamma(X)=2(X^2+X+1)$, and this polynomial has two simple roots.
		Therefore $s=4$.

		\emph{Third fact: if $s=4$, then}
		\[
		\gamma(X)=1+X^4+(1+X)^4=2(X^2+X+1)^2.
		\]
		This is a direct computation.
		So, for $s=4$, the roots of $\gamma$ are the two primitive cube roots of unity $\omega$ and $\omega^2$,
		and both have multiplicity two.

		\emph{Step 5. The functions $f_1,f_2,f_3$.}
		For $i=1,2,3$, put $f_i=y_i/y_4$. These are elements of $K^*$.
		We need their zeros and poles.

		\emph{The case $s\ne4$.}
		By Step~4, the polynomial $\gamma$ has a simple root $\lambda_0$.
		By Step~3, the following three points lie on $C$:
		\[
		P_1=[0:1:\lambda_0:-1-\lambda_0],\qquad
		P_2=[1:0:\lambda_0:-1-\lambda_0],\qquad
		P_3=[1:\lambda_0:0:-1-\lambda_0].
		\]
		At the point $P_i$, the coordinate $y_i$ is zero and the other three coordinates are not zero.
		By Step~3, the function $y_1/y_2$ has a zero of order one at $P_1$, and the functions $y_2/y_1$ and $y_3/y_1$
		have zeros of order one at $P_2$ and at $P_3$, respectively.
		Now $f_1=(y_1/y_2)\cdot(y_2/y_4)$, and the function $y_2/y_4$ is defined and not zero at $P_1$.
		So $\operatorname{ord}_{P_1}(f_1)=1$.
		In the same way, $\operatorname{ord}_{P_2}(f_2)=1$ and $\operatorname{ord}_{P_3}(f_3)=1$.
		For $j\ne i$, the function $f_j=y_j/y_4$ is defined and not zero at $P_i$.
		To sum up, for all $i,j\in\{1,2,3\}$ with $j\ne i$ we have
		\[
		\operatorname{ord}_{P_i}(f_i)=1,\qquad\operatorname{ord}_{P_i}(f_j)=0.
		\]
		We use the three points $P_1,P_2,P_3$ in Step~7.

		\emph{The case $s=4$.}
		By Step~4, the roots of $\gamma$ are $\omega$ and $\omega^2$, and both have multiplicity two.
		So, by Step~3, each of the four planes $y_i=0$ meets $C$ in exactly two points.
		We call them $P_i$ and $P_i'$.
		For example, since $-1-\omega=\omega^2$, the two points in the plane $y_4=0$ are
		\[
		P_4=[1:\omega:\omega^2:0],\qquad P_4'=[1:\omega^2:\omega:0].
		\]
		The eight points $P_i,P_i'$ are distinct, because a point of $C$ has at most one coordinate equal to zero.
		By Step~3, for $j\ne i$ the function $y_i/y_j$ has a zero of order two at $P_i$ and at $P_i'$.
		In other words, the plane $y_i=0$ meets $C$ with intersection multiplicity two at each of these two points.
		As in the case $s\ne4$, it follows that, for $i=1,2,3$, the function $f_i=y_i/y_4$ has zeros of order two at
		$P_i$ and at $P_i'$.
		The function $y_4/y_i$ has zeros of order two at $P_4$ and at $P_4'$. So $f_i$ has poles of order two at
		these two points.
		At every other point of $C$, the coordinates $y_i$ and $y_4$ are not zero, and $f_i$ has no zero and no pole.
		So the divisor of $f_i$ is
		\begin{equation}\label{eq:quartic-divisors}
			\operatorname{div}(f_i)=2P_i+2P_i'-2P_4-2P_4'.
		\end{equation}
		None of the $f_i$ is a square in $K$. If $f_i=g^2$ with $g\in K$, then $g$ is not
		constant, because $f_i$ is not constant.
		By \eqref{eq:quartic-divisors}, the divisor of $g$ is $P_i+P_i'-P_4-P_4'$.
		So $g$ is a nonconstant function with pole divisor $P_4+P_4'$, of degree two.
		This contradicts
		Remark~\ref{rem:plane-quartic}, because $C$ is a smooth plane quartic.
		On the other hand, Newton's identities \eqref{eq:newton-identities} for $k=2$ and $k=4$ are
		\[
		p_2=e_1p_1-2e_2,\qquad p_4=e_1p_3-e_2p_2+e_3p_1-4e_4.
		\]
		Since $p_1=e_1$, they give $p_2\equiv-2e_2$ and
		\[
		p_4\equiv2e_2^2-4e_4\pmod{p_1}.
		\]
		At the points of $C$ we have $p_1(y)=p_4(y)=0$, and therefore $e_2(y)^2=2e_4(y)$.
		Since $f_1f_2f_3=y_1y_2y_3y_4/y_4^4=e_4(y)/y_4^4$, we get
		\begin{equation}\label{eq:quartic-square-relation}
			f_1f_2f_3
			=\left(\frac{e_2(y)}{\sqrt2\,y_4^2}\right)^2.
		\end{equation}
		It follows that no product of two distinct $f_i$ is a square either:
		such a product, together with \eqref{eq:quartic-square-relation},
		would make the remaining $f_i$ a square. Thus the only subsets of
		$\{f_1,f_2,f_3\}$ whose product is a square are the empty subset and
		the full subset.

		\emph{Step 6. The curve $C_d$ and the algebra $B$.}
		Consider the map from $\mathbb P^3$ to $\mathbb P^3$ that is given by
		\[
		[x_1:x_2:x_3:x_4]\longmapsto[x_1^d:x_2^d:x_3^d:x_4^d].
		\]
		Since $p_1(x_1^d,\ldots,x_4^d)=p_d(x_1,\ldots,x_4)$ and $p_s(x_1^d,\ldots,x_4^d)=p_{sd}(x_1,\ldots,x_4)$,
		the curve $C_d$ of Step~1 is the inverse image of $C$ under this map.
		We compare the two curves on the open sets $y_4\ne0$ and $x_4\ne0$.

		On the set $y_4\ne0$ we may assume that $y_4=1$, and we use $y_1,y_2,y_3$ as coordinates.
		The points of $C$ with $y_4\ne0$ are then the points of $\C^3$ with
		\[
		y_1+y_2+y_3+1=0,\qquad y_1^s+y_2^s+y_3^s+1=0.
		\]
		Put
		\[
		A=\C[y_1,y_2,y_3]/(y_1+y_2+y_3+1,\ y_1^s+y_2^s+y_3^s+1).
		\]
		Because $(p_1,p_s)$ is a prime ideal that does not contain $y_4$, the ring $A$ is a domain, and it is the ring
		of polynomial functions on this part of $C$.
		Its field of fractions is $K$, and the image of $y_i$ in $A$ is the function $f_i$.

		In the same way, on the set $x_4\ne0$ we use the coordinates $z_i=x_i/x_4$ for $i=1,2,3$.
		The points of $C_d$ with $x_4\ne0$ are the points of $\C^3$ with
		\[
		z_1^d+z_2^d+z_3^d+1=0,\qquad z_1^{sd}+z_2^{sd}+z_3^{sd}+1=0.
		\]
		They form an affine curve. Put
		\[
		B_0=\C[z_1,z_2,z_3]/(z_1^d+z_2^d+z_3^d+1,\ z_1^{sd}+z_2^{sd}+z_3^{sd}+1).
		\]
		The irreducible components of this affine curve correspond to the minimal primes of $B_0$.
		By Step~1, no irreducible component of $C_d$ lies in the plane $x_4=0$.
		So every irreducible component of $C_d$ contains points with $x_4\ne0$.
		These points form an irreducible component of the affine curve, and every irreducible component of the
		affine curve comes in this way from exactly one irreducible component of $C_d$.
		Together with Step~1, this gives
		\[
		\#\{\text{minimal primes of }J\}=\#\{\text{minimal primes of }B_0\}.
		\]

		The two polynomials that define $B_0$ are obtained from the two polynomials that define $A$ when we replace
		$y_i$ by $z_i^d$. Hence
		\[
		B_0=A[z_1,z_2,z_3]/(z_1^d-f_1,\ z_2^d-f_2,\ z_3^d-f_3).
		\]
		So $B_0$ is a free $A$-module: the monomials $z_1^{u_1}z_2^{u_2}z_3^{u_3}$ with $0\leq u_i<d$ form a basis.
		Since $A$ is a domain, it follows that every nonzero element of $A$ is a nonzerodivisor on $B_0$.
		We now invert all nonzero elements of $A$. Then $A$ becomes its field of fractions $K$, and $B_0$ becomes the ring
		\[
		B=K[z_1,z_2,z_3]/(z_1^d-f_1,z_2^d-f_2,z_3^d-f_3).
		\]
		This is the algebra of Proposition~\ref{rem:root-components} with $r=3$, and it contains $B_0$ as a subring.
		The proposition can be used here, because the field $K$ contains $\C$, and so it contains all $d$th roots of unity.

		The prime ideals of $B$ correspond to the prime ideals $\mathfrak q$ of $B_0$ that contain no nonzero element
		of $A$, and this correspondence keeps inclusions.
		We show that these prime ideals $\mathfrak q$ are exactly the minimal primes of $B_0$.
		First, let $\mathfrak q$ be a minimal prime of $B_0$.
		Every element of a minimal prime of a Noetherian ring is a zerodivisor, and a nonzero element of $A$ is a
		nonzerodivisor on $B_0$. So $\mathfrak q$ contains no nonzero element of $A$.
		Second, $B$ is a finite product of fields, by Proposition~\ref{rem:root-components}.
		So no prime ideal of $B$ contains another prime ideal of $B$.
		Now let $\mathfrak q$ be a prime ideal of $B_0$ that contains no nonzero element of $A$.
		It contains a minimal prime $\mathfrak q'$ of $B_0$.
		Both $\mathfrak q$ and $\mathfrak q'$ correspond to prime ideals of $B$, and one of these contains the other.
		Hence $\mathfrak q=\mathfrak q'$, and $\mathfrak q$ is a minimal prime of $B_0$.

		A product of finitely many fields has as many prime ideals as it has factors. So we have proved that
		\[
		\#\{\text{minimal primes of }J\}=\#\{\text{field factors of }B\}.
		\]
		By Proposition~\ref{rem:root-components}, this number is the number of triples $(\alpha_1,\alpha_2,\alpha_3)$
		of integers with $0\leq\alpha_i<d$ such that $f_1^{\alpha_1}f_2^{\alpha_2}f_3^{\alpha_3}$ is a $d$th power in $K$.

		\emph{Step 7. Proof of {\rm(a)}.}
		Let $\alpha_1,\alpha_2,\alpha_3$ be integers with $0\leq\alpha_i<d$, and suppose that
		\[
		f_1^{\alpha_1}f_2^{\alpha_2}f_3^{\alpha_3}=\psi^d\qquad\text{for some }\psi\in K^*.
		\]
		We compare the orders of the two sides at the points $P_1,P_2,P_3$ of Step~5.
		The order of $\psi^d$ at $P_i$ is $d\operatorname{ord}_{P_i}(\psi)$.

		Suppose first that $s\ne4$. By Step~5, the order of the left side at $P_i$ is $\alpha_i$.
		So $\alpha_i=d\operatorname{ord}_{P_i}(\psi)$, and $d$ divides $\alpha_i$.
		Since $0\leq\alpha_i<d$, we get $\alpha_i=0$ for $i=1,2,3$.

		Suppose now that $s=4$. By \eqref{eq:quartic-divisors}, and because the eight points of Step~5 are distinct,
		we have $\operatorname{ord}_{P_i}(f_i)=2$ and $\operatorname{ord}_{P_i}(f_j)=0$ for $j\ne i$.
		So the order of the left side at $P_i$ is $2\alpha_i$, and
		\[
		2\alpha_i=d\operatorname{ord}_{P_i}(\psi)\qquad\text{for }i=1,2,3.
		\]
		If $d$ is odd, this shows that $d$ divides $\alpha_i$, and again $\alpha_i=0$ for $i=1,2,3$.

		So, in both cases of {\rm(a)}, the only triple is $(0,0,0)$.
		By Step~6, the ideal $J$ has exactly one minimal prime, and by Step~1 it is a prime ideal.
		This proves {\rm(a)}.

		\emph{Step 8. Proof of {\rm(b)}.}
		Let $s=4$, and let $d$ be even.
		We consider again the triples $(\alpha_1,\alpha_2,\alpha_3)$ of Step~7.
		The equation $2\alpha_i=d\operatorname{ord}_{P_i}(\psi)$ shows that $d/2$ divides $\alpha_i$.
		Since $0\leq\alpha_i<d$, each $\alpha_i$ is either $0$ or $d/2$.
		Let $\Lambda$ be the set of all $i\in\{1,2,3\}$ with $\alpha_i=d/2$. Then
		$f_1^{\alpha_1}f_2^{\alpha_2}f_3^{\alpha_3}=\bigl(\prod_{i\in\Lambda}f_i\bigr)^{d/2}$.
		We claim that
		\[
		\left(\prod_{i\in\Lambda}f_i\right)^{d/2}\in K^{*d}
		\quad\Longleftrightarrow\quad
		\prod_{i\in\Lambda}f_i\in K^{*2}.
		\]
		For the forward implication, write $\varphi=\prod_{i\in\Lambda}f_i$ and
		suppose that $\varphi^{d/2}=\psi^{d}$ for some $\psi\in K^*$. Then
		$(\varphi/\psi^2)^{d/2}=1$, so $\varphi/\psi^2$ is a root of unity. A root of unity in $K$ is algebraic over $\C$. So it lies in $\C$,
		and it has a square root in $\C$. Hence $\varphi$ is a square in $K$.
		The reverse implication follows by raising a square to the power $d/2$.
		By Step~5, there are exactly two subsets $\Lambda$ of $\{1,2,3\}$ for which $\prod_{i\in\Lambda}f_i$ is a
		square in $K$: the empty set and the set $\{1,2,3\}$.
		So exactly two triples $(\alpha_1,\alpha_2,\alpha_3)$ have the required property, namely $(0,0,0)$ and
		$(d/2,d/2,d/2)$.
		By Step~6, the ideal $J$ has exactly two minimal primes.

		It remains to find these two minimal primes. The ring $S/(p_d)$ is a domain
		by Theorem~\ref{rem:css-general}(b), applied with $r=1$.
		Put $F_\pm=E_2\pm\sqrt2\,G$ in this ring; that is, $F_+$ and $F_-$ also denote the images of the
		polynomials $E_2+\sqrt2\,G$ and $E_2-\sqrt2\,G$ in $S/(p_d)$. In Step~5 we saw that $p_4\equiv2e_2^2-4e_4$ modulo $p_1$.
		Now we replace every $x_i$ by $x_i^d$.
		Then $p_1$ becomes $p_d$, $p_4$ becomes $p_{4d}$, $e_2$ becomes $E_2$, and $e_4=x_1x_2x_3x_4$ becomes
		$(x_1x_2x_3x_4)^d=G^2$. So
		\[
		p_{4d}=2(E_2^2-2G^2)=2F_+F_-\quad\hbox{in }S/(p_d).
		\]
		It follows that $J=(p_d,F_+F_-)$. In particular, $J\subseteq Q_+$ and $J\subseteq Q_-$.
		For $d=2$, this factorization explains the zero-divisor of $S/(p_2,p_8)$ in \cite[Remark~3.3(6)]{CSS}.
		The image of $p_{4d}$ in $S/(p_d)$ is nonzero, since $(p_d,p_{4d})$
		has height two. The displayed factorization therefore shows that
		both $F_+$ and $F_-$ are nonzero in $S/(p_d)$. They are homogeneous of positive degree.

		We show that $Q_+$ and $Q_-$ are radical ideals.
		For this it is enough to show that the ideals $(F_+)$ and $(F_-)$ of $S/(p_d)$ are radical, because
		$S/Q_\pm$ is the quotient of $S/(p_d)$ by $(F_\pm)$.
		The principal ideal $(F_+F_-)$ of $S/(p_d)$ is radical, since $S/J$ is reduced.
		Let $\theta$ be an element of $S/(p_d)$, and suppose that $\theta^N\in(F_+)$ for
		some integer $N\geq1$. Then $(\theta F_-)^N\in(F_+F_-)$, and so $\theta F_-\in(F_+F_-)$. Cancellation of the nonzero element $F_-$ in
		the domain $S/(p_d)$ gives $\theta\in(F_+)$.
		So $(F_+)$ is radical. The same argument, with $F_+$ and $F_-$ exchanged, shows that $(F_-)$ is radical.

		Next, $p_d,F_+$ is a regular sequence of homogeneous elements of positive degrees, because $S/(p_d)$ is a
		domain and the image of $F_+$ in it is not zero.
		So every minimal prime of $Q_+$ has height two, by Remark~\ref{rem:regular-sequences}(b).
		Such a prime ideal contains $J$, and the minimal primes of $J$ also have height two.
		Hence every minimal prime of $Q_+$ is a minimal prime of $J$.
		The same holds for $Q_-$.

		No minimal prime of $J$ contains both $Q_+$ and $Q_-$.
		Indeed, a prime ideal that contains $Q_+$ and $Q_-$ contains $F_+-F_-=2\sqrt2\,G$.
		So it contains $x_1x_2x_3x_4$, and therefore one of the variables $x_i$.
		By Step~1, no minimal prime of $J$ contains a variable.

		We can now finish the proof.
		The ideal $J$ has exactly two minimal primes.
		Each of the ideals $Q_+$ and $Q_-$ has at least one minimal prime, these are minimal primes of $J$, and
		no minimal prime of $J$ contains both $Q_+$ and $Q_-$.
		So $Q_+$ has exactly one minimal prime, $Q_-$ has exactly one minimal prime, and these two prime ideals are
		the two minimal primes of $J$.
		A radical ideal with only one minimal prime is a prime ideal.
		Therefore $Q_+$ and $Q_-$ are prime ideals, and they are the two minimal primes of $J$.
		Since $J$ is radical, we get $J=Q_+\cap Q_-$.
	\end{proof}

	Theorem~\ref{thm:quartic-pair-components} gives domains that are not normal.
	In particular, we obtain Corollary~\ref{cor:C} of the Introduction; it is part~(a) of the next corollary.

	\begin{samepage}
		\begin{Cor}\label{cor:p3-domain}
			Let $S=\C[x_1,x_2,x_3,x_4]$.
			\begin{enumerate}[\rm(a)]
				\item Let $b>3$ be an integer. The ring $S/(p_3,p_b)$ is a domain that is not normal if and only if
				$b=18k+12$ for some integer $k\geq0$.
				\item For every integer $k\geq1$, the ring $S/(p_2,p_{12k+8})$ is a domain that is not normal.
			\end{enumerate}
		\end{Cor}
	\end{samepage}

	\begin{proof}
		We use \eqref{eq:pair-normality-gcd}: for $1<a<b$, the ring $S/(p_a,p_b)$ is a normal domain if and only if
		the integer $(b-a)/\gcd(a,b)$ is divisible neither by $2$ nor by $3$.

		(a) Let $b=18k+12$ with $k\geq0$. Then $b=3s$ with $s=6k+4$, and $s$ is even.
		Since $3$ is odd, the ideal $(p_3,p_b)$ is prime by Theorem~\ref{thm:quartic-pair-components}(a).
		Moreover, $(b-3)/\gcd(3,b)=(18k+9)/3=6k+3$ is divisible by $3$.
		So $S/(p_3,p_b)$ is a domain that is not normal.

		Conversely, let $b>3$ be an integer that is not of the form $18k+12$.
		If $b$ is odd, then $(p_3,p_b)$ is not a prime ideal, by \cite[Remark~3.3(1)]{CSS}; so $S/(p_3,p_b)$ is not a
		domain.
		Now let $b$ be even. Then $b-3$ is odd.
		If $3$ does not divide $b$, then $\gcd(3,b)=1$, and $3$ does not divide $b-3$.
		If $3$ divides $b$, write $b=3b'$. Then $b'$ is even, and $(b-3)/\gcd(3,b)=b'-1$.
		This number is divisible by $3$ exactly when $b'\equiv1\pmod3$. Since $b'$ is even, this means that
		$b'\equiv4\pmod6$, that is, $b=18k+12$ for some $k\geq0$, and this is not the case.
		So in both cases $(b-3)/\gcd(3,b)$ is divisible neither by $2$ nor by $3$, and $S/(p_3,p_b)$ is a normal domain.

		(b) Let $k\geq1$, and put $b=12k+8$. Then $b=2s$ with $s=6k+4$; here $s$ is even and $s\ne4$.
		So $(p_2,p_b)$ is a prime ideal by Theorem~\ref{thm:quartic-pair-components}(a).
		Moreover, $(b-2)/\gcd(2,b)=(12k+6)/2=6k+3$ is divisible by $3$.
		So $S/(p_2,p_b)$ is a domain that is not normal.
	\end{proof}

	\begin{remark}[Odd $s$]
		\label{rem:quartic-odd-s}
		Let $s\geq3$ be odd. Then the ideal $(p_d,p_{sd})$ is not prime; this follows from
		\cite[Remark~3.3(1) and~(2)]{CSS}.
		Indeed, choose $\zeta\in\C$ with $\zeta^d=-1$. Since $s$ is odd, we also have $\zeta^{sd}=-1$.
		So $p_d$ and $p_{sd}$ lie in the prime ideal $\mathfrak p=(x_2-\zeta x_1,\ x_4-\zeta x_3)$, which has height two.
		If $(p_d,p_{sd})$ were prime, it would be equal to $\mathfrak p$, because it also has height two.
		But $(p_d,p_{sd})$ does not contain the linear form $x_2-\zeta x_1$: it contains no linear form if $d\geq2$,
		and only the multiples of $p_1$ if $d=1$.
		For $s=3$ and $s=5$, the minimal primes of $(p_d,p_{sd})$ are given in
		Proposition~\ref{thm:triple-d-3d}(a) and Corollary~\ref{cor:triple-d-5d}(a).
	\end{remark}

	Together with Theorem~\ref{thm:quartic-pair-components}, this decides for all $d$ and $s$ whether
	$(p_d,p_{sd})$ is prime.

	\begin{Cor}\label{cor:pair-d-sd-prime}
		Let $d\geq1$ and $s\geq2$ be integers. The ideal $(p_d,p_{sd})$ of $S=\C[x_1,x_2,x_3,x_4]$ is prime if and
		only if $s$ is even and, in addition, $s\ne4$ or $d$ is odd.
	\end{Cor}

	\begin{proof}
		For even $s$, this is Theorem~\ref{thm:quartic-pair-components}.
		For odd $s$, the ideal is not prime by Remark~\ref{rem:quartic-odd-s}.
	\end{proof}

	\begin{note}
		Corollary~\ref{cor:pair-d-sd-prime} answers \cite[Question~2.1(5)]{CSS} for two power sums in four variables
		in the case where one of the two degrees divides the other.
	\end{note}

	\begin{remark}[The case $s=4$]
		\label{rem:quartic-s-four}
		For even $d$, the ideal $(p_d,p_{4d})$ is not prime \cite[Remark~3.3(2) and~(6)]{CSS}, and
		Theorem~\ref{thm:quartic-pair-components}(b) gives its two minimal primes.
		The proof of that theorem shows why $s=4$ is special: by Step~4, $4$ is the only even integer $s$ for which
		the polynomial $\gamma$ has no simple root.
	\end{remark}

	\begin{remark}[Normality]
		\label{rem:quartic-normality}
		Let $d\geq2$, and let $s$ be even. For $a=d$ and $b=sd$ we have $(b-a)/\gcd(a,b)=s-1$.
		So, by \eqref{eq:pair-normality-gcd}, the ring $S/(p_d,p_{sd})$ is a normal domain if and only if
		$\gcd(s-1,6)=1$. Since $s-1$ is odd, this means that $s\not\equiv4\pmod6$. Hence:
		\begin{enumerate}[\rm(a)]
			\item If $s\not\equiv4\pmod6$, then Theorem~\ref{thm:quartic-pair-components}(a) is already contained in
			Theorem~\ref{lem:pair-normality}(b), that is, in \cite[Theorem~3.1]{CSS}.
			\item If $s\equiv4\pmod6$, and if $s\ne4$ or $d$ is odd, then $S/(p_d,p_{sd})$ is a domain that is not
			normal, by Theorem~\ref{thm:quartic-pair-components}(a).
			For example, $S/(p_d,p_{4d})$ is such a ring for every odd $d\geq3$.
		\end{enumerate}
		For $d=1$ there is no such case: $S/(p_1,p_s)$ is a normal domain for every even $s$, by
		Theorem~\ref{lem:pair-normality}(a).
	\end{remark}

	\begin{note}
		The condition $s\equiv4\pmod6$ also appears in the proof of Theorem~\ref{thm:quartic-pair-components}:
		for even $s$, the polynomial $\gamma$ has a multiple root if and only if $s\equiv4\pmod6$.
		Indeed, by Step~4 of that proof, a multiple root is $\omega$ or $\omega^2$.
		Since $1+\omega=-\omega^2$ and $s$ is even, we have $\gamma(\omega)=1+\omega^s+\omega^{2s}$ and
		$\gamma'(\omega)=s(\omega^{s-1}-\omega^{2(s-1)})$.
		Both numbers are zero if and only if $3$ divides $s-1$, and the same holds for $\omega^2$.
		Compare \cite[Lemma~4.3]{CSS}, which gives the multiplicity of $\omega$ as a root of $1+z^n+(-1-z)^n$
		for every $n$.
	\end{note}

	\begin{remark}[Conjectures of Conca, Singh and Soundararajan]
		Corollary~\ref{cor:p3-domain}(a) proves the statement that is conjectured in \cite[Remark~3.3(7)]{CSS}.
		The conjecture in \cite[Remark~3.3(4)]{CSS} says that $S/(p_2,p_b)$ is a domain that is not normal exactly
		when $b=6k+5$ or $b=12k+8$, with $k\geq1$.
		Corollary~\ref{cor:p3-domain}(b) proves this for the degrees $b=12k+8$.
		Theorem~\ref{thm:quartic-pair-components} says nothing about the degrees $b=6k+5$, because $2$ does not
		divide them.
		For $k=0$, the ideal $(p_2,p_8)$ is not prime, by Remark~\ref{rem:quartic-s-four}.
		For every other degree $b>2$, the ring $S/(p_2,p_b)$ is either a normal domain or not a domain;
		this follows from Theorem~\ref{lem:pair-normality}(b) and \cite[Remark~3.3(3), (5) and~(6)]{CSS}.
		So the open part of this conjecture is whether $(p_2,p_{6k+5})$ is a prime ideal for every $k\geq1$.
	\end{remark}

	\begin{remark}[A conjecture of Kumar and Martino]
		The conjecture \cite[Conjecture~4.6(3)]{KM} predicts that $(p_3,p_{2m})$ is a prime ideal for every $m\geq1$
		with $m\ne6+9\lambda$, $\lambda\geq0$.
		In fact, $(p_3,p_b)$ is prime for every even $b>3$, also for the excluded degrees $b=18\lambda+12$:
		this follows from Theorem~\ref{lem:pair-normality}(b) and Corollary~\ref{cor:p3-domain}(a).
		For $b=2$ it is the ideal $(p_2,p_3)$, which is prime by Theorem~\ref{thm:power-sum-progressions}.
	\end{remark}

	\subsection*{Adding a third power sum}

	We first show that $p_c$ does not lie in the ideal $(p_1,p_4)$ if $c\geq2$ and $c\ne4$.
	Then we use the minimal primes of $(p_d,p_{4d})$ from Theorem~\ref{thm:quartic-pair-components}.
	We prove that $p_d,p_{4d},p_c$ is a regular sequence for every $d$ and for every degree $c$ different from
	$d$ and $4d$; the degree $c$ need not be a multiple of $d$.
	At the end of this subsection we consider the pairs $p_d,p_{sd}$ for the other even integers $s$.

	\begin{lemma}\label{lem:quartic-nonmembership}
		In $\C[x_1,x_2,x_3,x_4]$, we have $p_c\notin(p_1,p_4)$ for
		every integer $c\geq2$ with $c\ne4$.
	\end{lemma}

	\begin{proof}
		Suppose first that $c$ is even. Theorem~\ref{lem:pair-normality}(a) shows that both
		$(p_1,p_4)$ and $(p_1,p_c)$ are prime ideals.
		The ring $S/(p_1)$ is a polynomial ring in three variables, because we can replace $x_4$ by $-x_1-x_2-x_3$.
		The images of $p_4$ and $p_c$ in this ring are not zero, and each of them generates a prime ideal.
		So these images are irreducible polynomials, of degrees $4$ and $c$.
		Now suppose that $p_c\in(p_1,p_4)$. Then the image of $p_4$ divides the image of $p_c$.
		Since the image of $p_c$ is irreducible, the quotient is a nonzero constant.
		So the two polynomials have the same degree, and $c=4$. This contradicts $c\ne4$.
		Hence $p_c\notin(p_1,p_4)$ for every even $c\ne4$.

		Now let $c$ be odd, so that $c\geq3$. We have
		$p_1=e_1$, and $p_4\equiv2e_2^2-4e_4$ modulo $p_1$; see Step~5 in the proof of
		Theorem~\ref{thm:quartic-pair-components}.
		So $(p_1,p_4)=(e_1,\ e_4-e_2^2/2)$, and in the ring $S/(p_1,p_4)$ we have
		\[
		e_1=0,\qquad e_4=\frac{e_2^2}{2}.
		\]
		The images of $e_2$ and $e_3$ in $S/(p_1,p_4)$ are algebraically independent over $\C$.
		To see this, let $F$ be a polynomial in two variables with complex coefficients such that
		$F(e_2,e_3)\in(p_1,p_4)$.
		Choose arbitrary complex numbers $\lambda$ and $\mu$, and let $w_1,\ldots,w_4$ be the roots, with
		multiplicities, of the polynomial $X^4+\lambda X^2-\mu X+\lambda^2/2$.
		At the point $w=(w_1,\ldots,w_4)$ we have $e_1(w)=0$, $e_2(w)=\lambda$, $e_3(w)=\mu$ and $e_4(w)=\lambda^2/2$.
		So $p_1(w)=0$ and $p_4(w)=2\lambda^2-4\lambda^2/2=0$.
		Therefore every element of $(p_1,p_4)$ is zero at $w$, and $F(\lambda,\mu)=0$.
		This holds for all $\lambda$ and $\mu$. Hence $F$ is the zero polynomial.

		Set $D_0(t)=1+e_2t^2+e_2^2t^4/2$. By \eqref{eq:log-generating} with $n=4$, we have
		\[
		\sum_{j\geq1}\frac{p_j}{j}t^j=-\log D(t),\qquad\text{where }D(t)=1-e_1t+e_2t^2-e_3t^3+e_4t^4.
		\]
		In the ring $S/(p_1,p_4)$ we have $D(t)=D_0(t)-e_3t^3$, and therefore
		\[
		-\log D(t)=-\log D_0(t)-\log\left(1-\frac{e_3t^3}{D_0(t)}\right).
		\]
		The series $D_0(t)$ is invertible, because $D_0(0)=1$.
		We expand the last logarithm with the series $-\log(1-\varepsilon)=\sum_{k\geq1}\varepsilon^k/k$,
		where $\varepsilon=e_3t^3/D_0(t)$. In $S/(p_1,p_4)$ this gives
		\begin{equation}\label{eq:quartic-log-expansion}
			\sum_{j\geq1}\frac{p_j}{j}t^j
			=-\log D_0(t)+
			\sum_{k\geq1}\frac{e_3^k t^{3k}}{k D_0(t)^k}.
		\end{equation}
		We will find one nonzero coefficient of $p_c$. Here we write the image of $p_c$ in
		$S/(p_1,p_4)$ as a polynomial in $e_2$ and $e_3$. This is possible in only one way, because $e_2$ and $e_3$
		are algebraically independent.
		Since $D_0(t)$ contains no $e_3$, a monomial involving $e_3^k$ can
		come only from the summand with that value of $k$.
		Write $[z^u]\Phi(z)$ for the coefficient of $z^u$ in a
		formal power series $\Phi(z)$.
		We use this notation with $z=e_2t^2$; then $D_0(t)=1+z+z^2/2$.
		Since
		\[
		\frac1{1+z+z^2/2}=
		\frac{1-z+z^2/2}{1+z^4/4},
		\]
		the coefficient $[z^u](1+z+z^2/2)^{-1}$ is zero exactly when
		$u\equiv3\pmod4$.
		To check the displayed identity, note that $(1+z+z^2/2)(1-z+z^2/2)=(1+z^2/2)^2-z^2=1+z^4/4$.
		The series $(1+z^4/4)^{-1}=\sum_{j\geq0}(-1/4)^jz^{4j}$ contains only the powers $z^{4j}$, and their
		coefficients are not zero. The numerator is $1-z+z^2/2$.
		So the product contains exactly the powers $z^u$ with $u\equiv0,1,2\pmod4$.
		Taking $k=1$ in \eqref{eq:quartic-log-expansion}, we get the term $e_3t^3/D_0(t)$.
		It is the sum of the terms
		\[
		\bigl([z^u](1+z+z^2/2)^{-1}\bigr)\,e_3e_2^ut^{2u+3},\qquad u\geq0.
		\]
		We compare the coefficients of $t^c$ on both sides of \eqref{eq:quartic-log-expansion}, and we take
		$u=(c-3)/2$. We find that the coefficient
		of $e_3e_2^{(c-3)/2}$ in $p_c$ is
		\[
		c[z^{(c-3)/2}](1+z+z^2/2)^{-1}.
		\]
		This is nonzero unless $(c-3)/2\equiv3\pmod4$, that is, unless $c\equiv1\pmod8$.

		In the remaining case we have $c\equiv1\pmod8$ and $c\geq3$, hence $c\geq9$. We write $c=8\ell+9$, with $\ell\geq0$, and take
		$k=3$ instead. The term with $k=3$ in \eqref{eq:quartic-log-expansion} is
		$e_3^3t^9/\bigl(3D_0(t)^3\bigr)$, and the monomial $e_3^3e_2^{4\ell}$ occurs in it together with
		$t^{8\ell+9}=t^c$. So the coefficient of $e_3^3e_2^{4\ell}$ in $p_c$ is
		\[
		\frac c3[z^{4\ell}](1+z+z^2/2)^{-3}
		=\frac c3\frac{(-1)^\ell}{4^\ell}(\ell+1)(1-4\ell),
		\]
		which is nonzero, because $\ell$ is an integer. For completeness, the coefficient identity follows
		by writing
		\[
		(1+z+z^2/2)^{-3}
		=\frac{(1-z+z^2/2)^3}{(1+z^4/4)^3}.
		\]
		The denominator gives the series $(1+z^4/4)^{-3}=\sum_{j\geq0}\binom{j+2}{2}(-1/4)^jz^{4j}$, which contains
		only powers of $z^4$.
		Only the constant and fourth-degree terms of the numerator contribute
		to a degree divisible by four. Their coefficients are $1$ and $9/4$.
		The required coefficient is therefore
		\[
		\frac{(-1)^\ell}{4^\ell}
		\left\{\binom{\ell+2}{2}-9\binom{\ell+1}{2}\right\}
		=\frac{(-1)^\ell}{4^\ell}(\ell+1)(1-4\ell).
		\]
		Algebraic independence of $e_2,e_3$ now shows that the image of
		$p_c$ is nonzero in every odd degree $c\geq3$.
	\end{proof}

	\begin{Theorem}\label{thm:triple-d-4d}
		Let $d\geq1$ and $c\geq1$ be integers with $c\ne d$ and $c\ne4d$. Then $p_d,p_{4d},p_c$ form a regular
		sequence in $\C[x_1,x_2,x_3,x_4]$.
	\end{Theorem}

	\begin{proof}
		By Theorem~\ref{rem:css-general}(a), $p_d,p_{4d}$ is a regular sequence.
		By Remark~\ref{rem:regular-sequences}(d), the sequence $p_d,p_{4d},p_c$ is regular if and only if $p_c$ lies in
		no minimal prime of $(p_d,p_{4d})$.
		The minimal primes are given in Theorem~\ref{thm:quartic-pair-components}, with $s=4$:
		if $d$ is odd, the only minimal prime is $(p_d,p_{4d})$ itself; if $d$ is even, the minimal primes are
		$Q_+$ and $Q_-$.

		Suppose first that $d$ divides $c$, and write $c=dq$. Then $q\geq2$ and $q\ne4$.
		The sequence $p_d,p_{4d},p_{dq}$ is obtained from $p_1,p_4,p_q$ by the substitution $x_i\mapsto x_i^d$.
		By Lemma~\ref{lem:finite-free-substitution}(b), it is a regular sequence if and only if $p_1,p_4,p_q$ is a
		regular sequence.
		The ideal $(p_1,p_4)$ is prime by Theorem~\ref{lem:pair-normality}(a), and $p_q\notin(p_1,p_4)$ by
		Lemma~\ref{lem:quartic-nonmembership}.
		So $p_1,p_4,p_q$ is a regular sequence, by Remark~\ref{rem:regular-sequences}(d).

		Suppose now that $d\nmid c$. Let $\zeta$ be a primitive $(3d)$th root
		of unity, and put
		\[
		w=(1,\zeta,\zeta^2,0),\qquad \omega=\zeta^d.
		\]
		Then $\omega$ is a primitive cube root of unity, so
		$1+\omega+\omega^2=0$.
		Hence $p_d(w)=1+\omega+\omega^2=0$ and $p_{4d}(w)=1+\omega^4+\omega^8=1+\omega+\omega^2=0$.
		If $d$ is odd, the only minimal prime is $(p_d,p_{4d})$, and all its
		elements are zero at $w$.
		If $d$ is even, the two minimal primes are $Q_+$ and $Q_-$.
		We then have
		\[
		E_2(w)=\omega+\omega^2+\omega^3=0,\qquad G(w)=0.
		\]
		Here $E_2(w)=e_2(1,\omega,\omega^2,0)$, and $G(w)=0$ because the last coordinate
		of $w$ is zero.
		So the three polynomials $p_d$, $E_2+\sqrt2\,G$ and $E_2-\sqrt2\,G$ are zero at $w$.
		Hence all elements of $Q_+$ and all elements of $Q_-$ are zero at $w$.

		Finally, $p_c(w)=1+\zeta^c+\zeta^{2c}$ is nonzero. Otherwise
		$\zeta^c$ would be a root of the polynomial $1+X+X^2$, that is,
		a primitive cube root of unity. Then $\zeta^{3c}=1$.
		Since $\zeta$ is a primitive $(3d)$th root of unity, $3d$ divides $3c$, and so $d$ divides $c$.
		This is a contradiction.
		So $p_c$ is not zero at $w$, while every element of every minimal prime of $(p_d,p_{4d})$ is zero at $w$.
		Hence $p_c$ lies in none of these minimal primes, and $p_d,p_{4d},p_c$ is a regular sequence.
	\end{proof}

	Theorem~\ref{thm:triple-d-4d} answers \cite[Question~2.1(1)]{CSS} for three power sums whose degrees are
	$d$, $4d$ and a third integer $c$: they always form a regular sequence.
	By Remark~\ref{rem:regular-sequences}(d), the order of the three power sums does not matter.
	When $d$ is even, the ideal $(p_d,p_{4d})$ is not prime, and the proof uses its two minimal primes $Q_+$ and
	$Q_-$ from Theorem~\ref{thm:quartic-pair-components}.

	The point $w$ in the proof has a simple geometric meaning.
	The map $[x_1:x_2:x_3:x_4]\mapsto[x_1^d:x_2^d:x_3^d:x_4^d]$ sends $w$ to the point $[1:\omega:\omega^2:0]$ of
	the curve $C$ in the proof of Theorem~\ref{thm:quartic-pair-components}. This is one of the two points $P_4$
	and $P_4'$.
	When $d$ is even, $w$ lies on both irreducible components of the curve $C_d$. So the two components meet at $w$.

	For odd $d$, the case $d\nmid c$ of Theorem~\ref{thm:triple-d-4d} also follows from
	Lemma~\ref{lem:nonmultiple-degree}(b), because $(p_d,p_{4d})$ is then a prime ideal that is generated by
	polynomials in $x_1^d,\ldots,x_4^d$.
	For even $d$ this lemma is not enough: the generators $E_2\pm\sqrt2\,G$ of $Q_\pm$ are polynomials in
	$x_1^{d/2},\ldots,x_4^{d/2}$, but not in $x_1^d,\ldots,x_4^d$.
	We now use the same lemma for the other even integers $s$.

	\begin{Cor}\label{rem:quartic-triple-context}
		Let $d\geq1$ be an integer, let $s\geq2$ be an even integer, and let $c\geq1$ be an integer with $c\ne d$
		and $c\ne sd$.
		\begin{enumerate}[\rm(a)]
			\item If $d$ does not divide $c$, then $p_d,p_{sd},p_c$ form a regular sequence in $S$.
			\item If $c=dq$ with an integer $q$, then $p_d,p_{sd},p_c$ form a regular sequence in $S$ if and only if
			$p_q\notin(p_1,p_s)$. This condition holds if $q$ is even.
		\end{enumerate}
	\end{Cor}

	\begin{proof}
		By Theorem~\ref{rem:css-general}(a), $p_d,p_{sd}$ is a regular sequence.

		(a) For $s=4$, the statement is contained in Theorem~\ref{thm:triple-d-4d}. Let $s\ne4$.
		By Theorem~\ref{thm:quartic-pair-components}(a), the ideal $(p_d,p_{sd})$ is prime.
		It is generated by polynomials in $x_1^d,\ldots,x_4^d$, and it is different from $\m$, because its height is two.
		Since $d\nmid c$, Lemma~\ref{lem:nonmultiple-degree}(b) gives $p_c\notin(p_d,p_{sd})$.
		So $p_d,p_{sd},p_c$ is a regular sequence, by Remark~\ref{rem:regular-sequences}(d).

		(b) Let $c=dq$. Then $q\ne1$ and $q\ne s$.
		The sequence $p_d,p_{sd},p_{dq}$ is obtained from $p_1,p_s,p_q$ by the substitution $x_i\mapsto x_i^d$.
		By Lemma~\ref{lem:finite-free-substitution}(b), it is a regular sequence if and only if $p_1,p_s,p_q$ is a
		regular sequence.
		Since $s$ is even, $(p_1,p_s)$ is a prime ideal, by Theorem~\ref{lem:pair-normality}(a).
		So, by Remark~\ref{rem:regular-sequences}(d), $p_1,p_s,p_q$ is a regular sequence if and only if
		$p_q\notin(p_1,p_s)$.

		Now let $q$ be even. Then $(p_1,p_q)$ is also a prime ideal, by Theorem~\ref{lem:pair-normality}(a).
		As in the proof of Lemma~\ref{lem:quartic-nonmembership}, the images of $p_s$ and $p_q$ in the polynomial ring
		$S/(p_1)$ are irreducible polynomials. Their degrees $s$ and $q$ are different.
		So the image of $p_s$ does not divide the image of $p_q$, and $p_q\notin(p_1,p_s)$.
	\end{proof}

	In Corollary~\ref{rem:quartic-triple-context}(b), the case of an odd integer $q$ remains.
	For $s=2$, we have $p_q\notin(p_1,p_2)$ for every odd $q\geq3$ except $q=5$; this follows from
	Corollary~\ref{thm:triple-d-2d}.
	For $s=4$, we have $p_q\notin(p_1,p_4)$ for every odd $q\geq3$, by Lemma~\ref{lem:quartic-nonmembership}.
	For even $s\geq6$, we do not know whether $p_q\notin(p_1,p_s)$ for every odd $q\geq3$; see the end of this
	section.

	\subsection*{Summary and computations}

	We collect our conditions for three power sums in Table~\ref{tab:three-power-sums}. These results
	suggest a question for arbitrary triples, which we state below.
	Throughout, $S=\C[x_1,x_2,x_3,x_4]$, and $d\geq1$ and $c\geq1$ are integers.

	\begin{table}[ht]
	\centering
	\caption{Conditions for three power sums $p_d,p_{kd},p_c$ in four variables to form a regular sequence}
	\label{tab:three-power-sums}
	\begin{tabular}{@{}lll@{}}
	\hline
	Degrees & Values of $c$ & Condition for a regular sequence\\
	\hline
	$(d,2d,c)$ & $c\notin\{d,2d\}$ & $c\ne5d$\\
	$(d,3d,c)$ & $c\notin\{d,3d\}$ & $\nu_2(d)\ne\nu_2(c)$\\
	$(d,4d,c)$ & $c\notin\{d,4d\}$ & No condition\\
	$(d,5d,c)$ & $c\notin\{d,5d\}$ & $c\ne2d$ and $\nu_2(d)\ne\nu_2(c)$\\
	\hline
	\end{tabular}
	\end{table}
	The rows follow, respectively, from
	Corollary~\ref{thm:triple-d-2d},
	Proposition~\ref{thm:triple-d-3d}(c),
	Theorem~\ref{thm:triple-d-4d}, and
	Corollary~\ref{cor:triple-d-5d}(b).
	In all four rows, $c$ need not be a multiple of $d$, and $c$ need not be the largest of the three degrees.
	The order of the three power sums does not matter (Remark~\ref{rem:regular-sequences}(d)).
	So Table~\ref{tab:three-power-sums} decides, for every set of three distinct degrees that contains $d$ and $2d$, or $d$ and $3d$,
	or $d$ and $4d$, or $d$ and $5d$, whether the three power sums form a regular sequence.
	For the degrees $d$, $sd$ and $c$ with an even integer $s\geq6$, Corollary~\ref{rem:quartic-triple-context}
	gives a regular sequence if $d$ does not divide $c$, and also if $c/d$ is an even integer.

	Two conditions on the degrees are necessary
	for a regular sequence.
	Let $1\leq a<b<c$ be integers, and put $g=\gcd(a,b,c)$. First, suppose that the integers $a/g$, $b/g$ and $c/g$ are all
	odd. Choose $\zeta\in\C$ with $\zeta^g=-1$.
	Then all three power sums vanish on the plane
	$\{(u,\zeta u,v,\zeta v):u,v\in\C\}$, so their ideal has height
	at most two, and $p_a,p_b,p_c$ is not a regular sequence.
	This argument comes from \cite[Lemma~2.4(a)]{CKW}, where it is used in two variables: there $(1,\zeta)$ is a
	common zero of the power sums.
	Second, for every $d\geq1$, the power sums $p_d,p_{2d},p_{5d}$ do not form a regular sequence, by
	Corollary~\ref{thm:triple-d-2d}.

	These two conditions correct
	\cite[Conjecture~4.5]{KM}.
	For example, part~(3) of that conjecture predicts a regular sequence
	for $(2,6,10)$. But for this triple we have $g=2$, and the integers $a/g$, $b/g$, $c/g$ are $1$, $3$, $5$.
	They are all odd. So $p_2,p_6,p_{10}$ is not a regular sequence, by the first condition.

	The question below asks whether these two conditions are also sufficient.
	It is the analogue, for three power sums, of \cite[Conjecture~2.15]{CKW}, which is about four power sums in
	four variables.
	Condition~(b) of the question is condition~(3) of that conjecture, and condition~(a) corresponds to
	condition~(2) of that conjecture.
	That conjecture has one more condition, on the divisibility of the degrees by $2$, $3$ and $4$; the question
	contains no condition of this kind.
	By Table~\ref{tab:three-power-sums}, the answer to the question is yes for every triple that contains $d$ and $2d$, or $d$ and $3d$,
	or $d$ and $4d$, or $d$ and $5d$, for some integer $d\geq1$.
	We do not know the answer in general.

	\begin{question}\label{question:triple-normalized-classification}
	Let $1\leq a<b<c$ be integers, and put $g=\gcd(a,b,c)$.
	Do $p_a,p_b,p_c$ form a regular sequence in $S$ if and only if
	both of the following conditions hold?
	\begin{enumerate}[\rm(a)]
	\item The integers $a/g,b/g,c/g$ are not all odd.
	\item There is no positive integer $d$ such that $(a,b,c)=(d,2d,5d)$.
	\end{enumerate}
	\end{question}

	For the degrees $d$, $sd$ and $c$ with an even integer $s\geq4$, the question can be made more precise.
	Conditions (a) and (b) hold for all these triples.
	Indeed, the integer $sd/g$ is even; and the quotient of the two degrees $sd$ and $d$ is $s$, which is different
	from $2$, $5$ and $5/2$, the quotients of two of the degrees $d$, $2d$, $5d$.
	So the question asks whether $p_d,p_{sd},p_c$ is a regular sequence for every $c$ different from $d$ and $sd$.
	By Corollary~\ref{rem:quartic-triple-context}, this is true if and only if $p_q\notin(p_1,p_s)$ for every odd
	integer $q\geq3$.
	For $s=4$ this is true, by Lemma~\ref{lem:quartic-nonmembership}.
	For even $s\geq6$ it is open.

	Computer calculations (CoCoA \cite{CoCoA}, Macaulay2 \cite{M2}) suggest that the answer to
	Question~\ref{question:triple-normalized-classification} is yes, and that $p_q\notin(p_1,p_s)$ for every even
	integer $s\geq6$ and every odd integer $q\geq3$.
	By Corollary~\ref{rem:quartic-triple-context}, the second statement would mean that $p_d,p_{sd},p_c$ is a regular
	sequence for every $c$ different from $d$ and $sd$.
	Our computations were done with CoCoA-5 \cite{CoCoA} and repeated independently with Macaulay2 \cite{M2}.
	To test whether $p_a,p_b,p_c$ form a regular sequence, we used the height criterion
	(Remark~\ref{rem:regular-sequences}(b)) and computed modulo a large prime number.
	This is enough: if power sums form a regular sequence modulo a prime number, then they also form a regular
	sequence over $\C$, because the dimension of their common zero set can only become larger when we reduce
	modulo a prime.
	To test whether $p_q\in(p_1,p_s)$, we used that $S/(p_1)$ is a polynomial ring in three variables, in which
	the image of $(p_1,p_s)$ is the principal ideal generated by the image of $p_s$.

	\medskip
	\noindent{\bf Acknowledgments.}
	The author thanks his Ph.D. advisor, Aldo Conca, for many helpful discussions on the topics of this paper.

	\vspace{2mm}
	\noindent{\bf Statements and Declarations.}
	The author declares no competing interests relevant to this article.

\end{document}